\documentclass[12pt,a4paper]{article}
\usepackage[utf8]{inputenc}
\usepackage[affil-it]{authblk}
\usepackage{amsmath}
\usepackage{amsthm}
\usepackage{bbm}
\usepackage[top=3cm,left=3cm,right=3cm,bottom=3cm]{geometry}
\usepackage{txfonts}
\usepackage{amsfonts}
\usepackage{amssymb}
\usepackage{enumerate}
\usepackage{enumitem}
\usepackage{verbatim}
\usepackage{graphicx}
\newcommand{\ddim}{{d}} % space dimension >= 3
\newcommand{\cstNewtonLp}[1]{{c_{{#1},\ddim}}} % const in the L^p estimate of F
\newcommand{\cstNewtonLn}{c_{\ddim}} % const in the ln estimate of naX F
\newcommand{\cstMoment}[1]{c_{{#1}}} % const in the estimate of local and global moments
\newcommand{\volS}[1]{\tau_{{#1}}} % volume of S^d
\newcommand{\surS}[1]{\omega_{{#1}}} % surface of S^d
\newcommand{\newtonian}[1]{{\Gamma\left({#1}\right)}}
\newcommand{\ball}[2]{B\left({#1};{#2}\right)}
\newcommand{\EvoOp}[2]{\mathfrak{U}_{{#1},{#2}}}

\newcommand{\minR}{A}
\newcommand{\globalBound}[2]{\mathfrak{B}_{M,\kappa,{#1}}^{{#2}}}

\newcommand{\NN}{\mathbb N}

\newcommand{\RN}{\mathbb R}
\newcommand{\CN}{\mathbb C}

\newcommand{\mi}{\mathrm i}

\newcommand{\SP}[2]{\left\langle {#1} , {#2} \right\rangle}
\newcommand{\norm}[1]{\left|\left|{#1}\right|\right|}

\newcommand{\symp}[2]{\left[{#1},{#2}\right]}

\newcommand{\sympf}[2]{\omega\left({#1},{#2}\right)}

\newcommand{\Lp}[2]{\mathcal{L}^{{#1}}_{{#2}}}
\newcommand{\Lpn}[3]{\norm{{#1}}_{\Lp{{#2}}{{#3}}}}

\newcommand{\Wp}[2]{\mathcal{W}^{{#1}}_{{#2}}}

\newcommand{\Cp}[2]{\mathcal{C}^{{#1}}_{{#2}}}

\newcommand{\Ap}[2]{\mathcal{A}^{{#1}}_{{#2}}}
\newcommand{\Apn}[3]{\norm{{#1}}_{\Ap{{#2}}{{#3}}}}
\newcommand{\Bp}[2]{\mathcal{B}^{{#1}}_{{#2}}}
\newcommand{\Bpn}[3]{\norm{{#1}}_{\Bp{{#2}}{{#3}}}}

\newcommand{\id}{\textrm{id}}
\newcommand{\unity}{\ensuremath{\mathbbm{1}}}
\newcommand{\intd}{\mathrm d}
\newcommand{\abs}[1]{\left|{#1}\right|}

\newcommand{\convol}[2]{\left({#1}*{#2}\right)}
\newcommand{\moment}[2]{\mathfrak{m}_{{#2}}\left({{#1}}\right)}
\newcommand{\Moment}[2]{\mathfrak{M}_{{#2}}\left({{#1}}\right)}

\newcommand{\Cm}[2]{\mathcal C^{#1}\left({#2}\right)}
\newcommand{\Ccm}[2]{\mathcal C_c^{#1}\left({#2}\right)}

\newcommand{\supp}{\text{supp}}

\newcommand{\HamVec}[1]{X_{{#1}}}
\newcommand{\HF}{\mathcal{H}}
\newcommand{\UF}{\mathcal{U}}
\newcommand{\HVF}{\mathcal{H}_{\text{Vl.}}}
\newcommand{\RX}{{\RN^{\ddim}_{\bx}}}
\newcommand{\RV}{{\RN^{\ddim}_{\bv}}}
\newcommand{\RZ}{{\RN^{2\ddim}_{\bz}}}

\newcommand{\RNp}{\RN_{\geq 0}}
\newcommand{\naX}{\nabla_{\bx}}
\newcommand{\naV}{\nabla_{\bv}}
\newcommand{\naZ}{\nabla_{\bz}}

\newcommand{\naW}{\nabla_{\bw}}
\newcommand{\phF}[1]{K_{{#1}}}

\newcommand{\del}[2]{\frac{\partial {#1}}{\partial {#2}}}
\newcommand{\Del}[1]{\mathrm D^{#1}}
\newcommand{\dKer}[3]{{#1}_{{#2}}^{({#3})}}

\newcommand{\bx}{\textbf x}
\newcommand{\bv}{\textbf v}
\newcommand{\bw}{\textbf w}
\newcommand{\bz}{\textbf z}
\newcommand{\by}{\textbf y}

\newcommand{\bn}{\textbf 0}

\newcommand{\bound}[2]{\mathfrak{b}_{{#1}}^{{#2}}}
\makeindex

\begin{document}

% Environment settings

\theoremstyle{plain}
\newtheorem{thm}{Theorem}[section]
\newtheorem{lem}[thm]{Lemma}
\newtheorem{cor}[thm]{Corollary}
\newtheorem{prop}[thm]{Proposition}
\newtheorem{conj}[thm]{Conjecture}

\theoremstyle{definition}
\newtheorem{defn}[thm]{Definition}

\theoremstyle{remark}
\newtheorem{rmk}[thm]{Remark}
\newtheorem{exam}[thm]{Example}

\title{Generalized symplectization of Vlasov dynamics and application to the Vlasov-Poisson system}
\author{R.A. Neiss
	\thanks{
		Electronic address: \texttt{rneiss@math.uni-koeln.de}
	}
}
\affil{
	Universität zu Köln, Mathematisches Institut \\
	Weyertal 86-90, 50931 Köln, Germany
}
\date{June 4, 2018}

\maketitle

\begin{abstract}
\noindent In this paper, we study a Hamiltonian structure of the Vlasov-Poisson system, first mentioned by Fröhlich, Knowles, and Schwarz \cite{froehlichknowlesschwarz}. To begin with, we give a formal guideline to derive a Hamiltonian on a subspace of complex-valued $\Lp{2}{}$ integrable functions $\alpha$ on the one particle phase space $\RZ$, s.t. $f=\abs{\alpha}^2$ is a solution of a collisionless Boltzmann equation. The only requirement is a sufficiently regular energy functional on a subspace of distribution functions $f\in\Lp{1}{}$. Secondly, we give a full well-posedness theory for the obtained system corresponding to Vlasov-Poisson in $\ddim\geq3$ dimensions. Finally, we adapt the classical globality results \cite{lionsperthame,pfaffelmoser,schaeffer} for $\ddim=3$ to the generalized system.
\end{abstract}

\section{Overview and motivation}

Vlasov dynamics is an important possibility to describe collective behavior of many particle systems by reducing the information to a distribution on the one particle phase space. The resulting evolution equation has been analyzed for a huge number of different physical systems, e.g., interactions, symmetries, boundary conditions.

To make a connection to finite-dimensional Hamiltonian dynamics, the conjugate variable to the phase space distribution function $f\in\Lp{1}{}$ has been missing. Nevertheless, symplectization of dynamics is a relevant question, as it could make accessible methods from finite-dimensional theory. 

Early contributions were made by Morrison \cite{morrison}, as well as Marsden and Weinstein \cite{marsdenweinstein}, who introduced a Lie-Poisson bracket for the Maxwell-Vlasov system. 

Ye, Morrison, and Crawford \cite{yemorrisoncrawford} adapted these findings and formally derived a Lie-Poisson bracket for the linear space of distribution functions of the Vlasov-Poisson system. In order to construct a symplectic framework, they hope to remove degeneracy of this bracket by foliating the linear space into symplectic leaves. They are the orbits of a symplectomorphism group acting by composition. 

Conjecturing that all these symplectomorphisms can be obtained as time-1 maps of autonomous Hamiltonian systems, they compute a Poisson bracket restricted on the leaf, that is claimed to no longer be degenerate. They expect that this non-de\-generate bracket is the cosymlectic form of the Kirillov-Kostant-Souriau form from finite-dimensional theory, rendering the leaf a symplectic manifold. 

As intriguing as these ideas are, rigorous verification of the formal computation still seems to be missing. In addition, the derived bracket on the symplectic leaf is by its implicit construction unhandy to solve rigorous problems in a PDE setup. It seems that no such results have arisen from it. \\

\noindent A more promising ansatz occurs in the work of Fröhlich, Knowles, and Schwarz as starting point to compute mean-field limits of bosonic quantum systems \cite[Sec.2]{froehlichknowlesschwarz}. Therein, letting the Vlasov density $f=\abs{\alpha}^2$, the symplectic form $\sympf{\alpha}{\beta} \equiv \Im \SP {\alpha}{\beta}_{\Lp{2}{\bz}}$ on $\Lp{2}{\bz}$ allows to find a new Hamiltonian $\HVF(\alpha)$, that no longer can be seen as the system's energy, but leads to a Hamiltonian evolution equation for $\alpha$, that reproduces the Vlasov equation for $\abs{\alpha}^2$. In their work, they treat the explicit case of non-relativistic dynamics with velocity-independent two-body interaction potential. \\

\noindent Inspired by \cite[Sec.2]{froehlichknowlesschwarz}, we intend to give a formal concept for symplectization of any collisionless Boltzmann equation where the starting point is the energy functional $\HF(f)$, $f$ a non-negative distribution function on the one particle phase space. We solve the formally obtained equations for the example of Vlasov-Poisson, i.e., non-relativistic motion in Newtonian two-body interaction, and give a full well-posedness theory for it, that is just vaguely discussed in \cite{froehlichknowlesschwarz}. At last, we shortly adapt the classical global solution results \cite{lionsperthame,pfaffelmoser,schaeffer}, whose globality criteria due to the close relation $f=\abs{\alpha}^2$ almost coincide with ours.

\section{Conceptual outline}

In this section, we formally derive a Hamiltonian evolution equation for a general collisionless Boltzmann equation for a single particle type with non-negative density $f$. We remark that the concept is easily generalized to multiple particle types $\{f_\nu\}$, but we omit the details for the sake of simplicity in notation.

\subsection{General Vlasov dynamics}

The starting point is the well-known derivation of the classical Vlasov equation. For the most general setup we consider a distribution function $f: \RZ \to\RN$ on the single particle phase space $\RZ=\RX\times\RV$. We assume that the full system's energy functional depending on the density $f$ is given by $\HF(f)\in\RN$; in the non-relativistic case we have in mind
\begin{equation}
\HF (f) \equiv \int_{\RZ} \frac{\abs{\bv}^2}{2} f(\bz)~ \intd\bz + \UF(f).
\end{equation}
$\UF(f)$ is some type of potential energy term. If we only have two-particle interaction that does not depend on velocity, e.g.,
\begin{equation}
\UF (f) = \frac{1}{2}\int_{\RZ} \int_{\RZ} \Gamma(\bx_1,\bx_2)~ f(\bz_1)~ f(\bz_2)~ \intd\bz_1~ \intd\bz_2
\end{equation}
for some symmetric two particle interaction potential $\Gamma:\RX\times\RX\to\RN$. We restrict our general choice of $\HF(f)$, by the requirement that
\begin{equation}
\label{eqn:functional-derivative}
\Del{1}\HF(f)[\delta f] \equiv \lim_{\epsilon\to 0} \frac{1}{\epsilon} \left[\HF(f+\epsilon~ \delta f) - \HF(f)\right] = \int_{\RZ} \dKer{H}{f}{1}(\bz)~ \delta f(\bz)~ \intd\bz
\end{equation}
exists in some sense for a kernel function $\dKer{H}{f}{1}: \RZ\to\RN$. Again, have the special form
\begin{equation*}
\dKer{H}{f}{1}(\bz) = \frac{\abs{\bv}^2}{2} + \dKer{U}{f}{1}(\bx), \quad \dKer{U}{f}{1}(\bx) = \int_{\RZ} \Gamma(\bx,\bar{\bx})~ f(\bar{\bz})~ \intd\bar{\bz}
\end{equation*}
in mind. $\dKer{H}{f}{1}$ can physically be interpreted as the point mass 1 single particle autonomous Hamiltonian for a static background distribution $f$. Of course, one can replace $f$ by some time dependent $f(t)$, which will render the Hamiltonian non-autonomous. \\

\noindent Actually, in the derivation of classical Vlasov dynamics, we now will have to solve the self-consistent system of characteristic equations
\begin{align}
\nonumber
\partial_t X(t,s,\bz) =& \left(\naV \dKer{H}{f(t)}{1}\right)(X(t,s,\bz),V(t,s,\bz)), \\
\label{eqn:general-characteristic-system}
\partial_t V(t,s,\bz) =& \left(-\naX \dKer{H}{f(t)}{1}\right)(X(t,s,\bz),V(t,s,\bz))
\end{align}
with initial condition $(X(t=s,s,\bz),V(t=s,s,\bz))=\bz=(\bx,\bv)$ and $f(t,\bz) = \mathring{f}(X(0,t,\bz),V(0,t,\bz))$. This set of equations can be reduced to the transport equation
\begin{align}
\nonumber
0 =&~ \partial_t f(t,\bz) + \left(\naV \dKer{H}{f(t)}{1}\right)(\bz) \cdot \naX f(t,\bz) + \left(-\naX \dKer{H}{f(t)}{1}\right)(\bz) \cdot \naV f(t,\bz) \\
\label{eqn:general-vlasov}
=&~ \partial_t f(t,\bz) + \symp{f(t)}{\dKer{H}{f(t)}{1}}(\bz).
\end{align}
with initial data $f(0)=\mathring{f}$. $\symp{\alpha}{\beta}(\bz) \equiv \naX\alpha(\bz) \cdot \naV\beta(\bz) - \naV\alpha(\bz) \cdot \naX\beta(\bz)$ is the standard Poisson bracket for differentiable functions $\alpha, \beta:\RZ=\RX\times\RV\to\CN$. \eqref{eqn:general-vlasov} is known as the Vlasov equation which has been studied for many different potentials and even single particle phase spaces. 

\subsection{Guideline: symplectization of Vlasov dynamics}

Now our focus shifts towards constructing a Hamiltonian evolution equation from some Vlasov equation. The technical assumptions on the Hamiltonian are $\HF(f)\in\RN$ and it is twice differentiable in some sense, i.e., not only $\Del{1} \HF(f) [\delta f]$, but also
\begin{align*}
\Del{2}\HF(f)[\delta f_1,\delta f_2] \equiv& \lim_{\epsilon\to 0} \frac{1}{\epsilon} \left[\Del{1}\HF(f+\epsilon~\delta f_1)[\delta f_2] - \Del{1}\HF(f)[\delta f_2]\right] \\
=& \int_{\RZ} \int_{\RZ} \dKer{H}{f}{2}(\bz_1,\bz_2)~ \delta f_1(\bz_1)~ \delta f_2(\bz_2)~ \intd\bz_1~\intd\bz_2
\end{align*}
exists. We will describe a canonical way to find a complex valued Vlasov type equation, that we expect to allow adaptation of methods from symplectic geometry. \\

\noindent In most applications, it is physically reasonable to assume that the overall particle mass is bounded, i.e., $f\in\Lp{1}{\bz}$, and non-negative (otherwise consider $f=f^+-f^-$ with the same characteristic equations). If we now take some $\alpha:\RZ\to\CN$, s.t. $f=\abs{\alpha}^2$, we can exploit that $\Lp{2}{\bz}$ of complex valued functions allows to construct the translation invariant symplectic form
\begin{equation}
\sympf{\alpha}{\beta} \equiv \Im \int_{\RZ} \alpha(\bz)~ \bar{\beta}(\bz)~\intd\bz.
\end{equation}
It is now possible to define (on a set of suitable functions $\alpha\in\Lp{2}{\bz}$) the new \textbf{Hamiltonian Vlasov functional}
\begin{align*}
\HVF(\alpha) \equiv~ \frac{1}{2} \Im~ \Del{1}\HF\left(\abs{\alpha}^2\middle)\middle[\symp{\bar{\alpha}}{\alpha}\right] = \frac{1}{2} \Im\int_{\RZ} \dKer{H}{\abs{\alpha}^2}{1}(\bz)~ \symp{\bar{\alpha}}{\alpha}(\bz)~ \intd\bz.
\end{align*}
Now, formal computation of the derivative shows
\begin{align*}
\Del{1}\HVF(\alpha)[\delta\alpha] =&\lim_{\epsilon\to 0} \frac{1}{\epsilon} \left[\HVF(\alpha+\epsilon~ \delta\alpha)-\HVF(\alpha)\right] \\
=&~ \frac{1}{2} \Im~ \Del{1}\HF\left(\abs{\alpha}^2\middle)\middle[2\mi~\Im\symp{\delta\bar{\alpha}}{\alpha}\right] + \frac{1}{2} \Im~ \Del{2}\HF\left(\abs{\alpha}^2\middle)\middle[2\Re~\left(\alpha~ \delta\bar{\alpha}\right), \symp{\bar{\alpha}}{\alpha}\right] \\
=&~ \Im\int_{\RZ} \left(\symp{\alpha}{\dKer{H}{\abs{\alpha}^2}{1}} (\bar{\bz}) + \alpha(\bar{\bz})~ \underbrace{\int_{\RZ} \symp{\bar{\alpha}}{\alpha}(\bz)~ \dKer{H}{\abs{\alpha}^2}{2}(\bar{\bz},\bz)~ \intd\bz}_{\equiv \phF{\alpha}(\bar{\bz})}\right)~ \delta\bar{\alpha}(\bar{\bz})~ \intd\bar{\bz} \\
\equiv& -\sympf{\HamVec\HVF(\alpha)}{\delta\alpha},
\end{align*}
defining the Hamiltonian vector field $\HamVec{\HVF}(\alpha)$ in the usual manner. Finally, it is possible to write down a \textbf{Hamiltonian} evolution equation for $\alpha$
\begin{align}
\label{eqn:general-hamilton-vlasov}
\partial_t\alpha(t,\bz) =&~ \HamVec\HVF(\alpha(t))(\bz) = \symp{\dKer{H}{\abs{\alpha(t)}^2}{1}}{\alpha(t)}(\bz) - \phF{\alpha(t)}(\bz)~ \alpha(t,\bz) \\
\nonumber
=&~ -\naX\alpha(t,\bz) \cdot \left(\naV\dKer{H}{\abs{\alpha(t)}^2}{1}\right)(\bz) - \naV\alpha(t,\bz) \cdot \left(-\naX\dKer{H}{\abs{\alpha(t)}^2}{1}\right)(\bz) - \phF{\alpha(t)}(\bz)~ \alpha(t,\bz),
\end{align}
where the function $\phF{\cdot}$ is seen to be imaginary valued. This is extremely important, as it renders the following formal result.

\begin{prop}
\label{prop:hvl-solves-vl}
Let $\alpha: I\times\RZ\to\CN$ be a local solution of \eqref{eqn:general-hamilton-vlasov}. Then $f(t,\bz)\equiv\abs{\alpha(t,\bz)}^2$ is a local solution of the classical Vlasov equation \eqref{eqn:general-vlasov}.
\end{prop}
\begin{proof}
Multiply \eqref{eqn:general-hamilton-vlasov} with $2\bar{\alpha}$ and apply $\Re$. Upon noting $2\Re~\bar{\alpha}~ \partial\alpha = \partial\abs{\alpha}^2$ for any first order derivative $\partial$ and $\Re~ \phF{\alpha}=0$, we recover \eqref{eqn:general-vlasov} for $f=\abs{\alpha}^2$. \qedhere
\end{proof}

\noindent A second formal result is a transport formula for the solution of the Hamiltonian equation \eqref{eqn:general-hamilton-vlasov}.

\begin{prop}
\label{prop:general-transport-formula}
Let $\alpha:I\times\RZ\to\CN$ be some function, $0\in I$. It is a local solution of \eqref{eqn:general-hamilton-vlasov} if and only if
\begin{equation}
\label{eqn:general-transport-formula}
\alpha(t,\bz) = \alpha(0,Z(0,t,\bz))~ \exp\left(-\int_{0}^{t} \phF{\alpha(\tau)}(Z(\tau,t,\bz))~ \intd\tau\right),
\end{equation}
where $Z=(X,V)$ is the solution map of the general characteristic system \eqref{eqn:general-characteristic-system}.
\end{prop}
\begin{proof}
\textbf{$\Rightarrow$.} Let $\alpha$ be a local solution and $Z$ the solution map of the characteristic system. For any $t,s\in I$, we find
\begin{align*}
\intd_t (\alpha(t,Z(t,s,\bz))) =&~ (\partial_t\alpha)(t,Z(t,s,\bz)) + \symp{\alpha}{\dKer{H}{\abs{\alpha(t)}^2}{1}}(Z(t,s,\bz)) \\
\stackrel{\eqref{eqn:general-hamilton-vlasov}}{=}&~ -\phF{\alpha(t)}(Z(t,s,\bz))~ \alpha(t,Z(t,s,\bz)),
\end{align*}
which is a linear ODE with unique solution for fixed $s,\bz$:
\begin{equation*}
\alpha(t,Z(t,s,\bz)) = \alpha(0,Z(0,s,\bz))~ \exp\left(-\int_{0}^{t} \phF{\alpha(\tau)}(Z(\tau,s,\bz))~\intd\tau\right).
\end{equation*} \\

\noindent\textbf{$\Leftarrow$.} Let $\alpha$ satisfy \eqref{eqn:general-transport-formula} and $Z$ be the solution map of \eqref{eqn:general-characteristic-system}. We then compute
\begin{align*}
\alpha(t,Z(t,0,\bz)) = \alpha(0,\bz)~ \exp\left(-\int_{0}^{t} \phF{\alpha(\tau)} (Z(\tau,0,\bz))~ \intd\tau\right).
\end{align*}
Differentiating both sides w.r.t. $t$ yields exactly \eqref{eqn:general-hamilton-vlasov} again. \qedhere
\end{proof}

\noindent It is reasonable to assume that \eqref{eqn:general-transport-formula} provides the structure to solve the Hamilton-Vlasov system with an iterative scheme. This is rigorously elaborated for the Vlasov-Poisson system in Section \ref{sec:vlasov-poisson}.

\subsection{Degeneracy of Hamiltonian structures}

It is worth noting that the choice of a Hamiltonian functional is not unique\footnote{The author thanks one of the referees for pointing this out}. In fact, if we assume that $\mathcal{G}: \Lp{1}{\bz} \to \RN$ is any map with first functional derivative in the sense of \eqref{eqn:functional-derivative}, then
\begin{equation*}
\widetilde{\HVF}(\alpha) \equiv \HVF(\alpha) + \mathcal{G}\left(\abs{\alpha}^2\right)
\end{equation*}
yields another Hamiltonian whose trajectories also solve \eqref{eqn:general-hamilton-vlasov} with a different imaginary valued phase term $K$. While this degeneracy of the problem might offer exciting opportunities for stability analysis or related issues, for the questions raised in this paper, the canonical choice $\mathcal{G}\equiv 0$ seems sufficient. \\

\noindent We close this section by formally classifying all these Hamiltonians which also lead to solutions of \eqref{eqn:general-hamilton-vlasov}.

\begin{prop}
\label{prop:gaugeing}
Let $\HVF: \Lp{2}{\bz}\to\RN$ be given as above. Then
\begin{equation*}
\widetilde{\HVF}(\alpha) \equiv \HVF(\alpha) + \mathcal{G}\left(\alpha\right)
\end{equation*}
for a smooth functional $\mathcal{G}:\Lp{2}{\bz} \to \CN$ generates trajectories satisfying Proposition \ref{prop:hvl-solves-vl} if and only if
\begin{equation*}
\mathcal{G}(\alpha) = \tilde{\mathcal{G}}\left(\abs{\alpha}^2\right)
\end{equation*}
for some smooth $\tilde{\mathcal{G}}: \Lp{1}{\bz} \to \RN$.
\end{prop}
\begin{proof}
At first, we note that $\mathcal{G}(\alpha)\in\RN$, because it would not even generate a Hamiltonian vector field otherwise. Now pick any $\alpha\in\Lp{2}{\bz}$ from the dense subset where $\left\{\alpha=0\right\}$ is a null set. With $U_\gtrless\equiv\{\Im\frac{\dKer{G}{\alpha}{1}}{\bar{\alpha}} \gtrless 0\}$,
\begin{equation*}
0 \leq \int_{U_>} \Im \frac{\dKer{G}{\alpha}{1}(\bz)}{\bar{\alpha}(\bz)}~ \abs{\alpha(\bz)}^2 \intd\bz = \Im\int \dKer{G}{\alpha}{1}(\bz)~ \left(\unity_{U_>} \alpha\right)(\bz)~ \intd\bz = \Im \left(\Del{1}\mathcal{G}(\alpha)\left(\unity_{U_>}\alpha\right)\right) = 0,
\end{equation*}
yielding that $U_>$ and equally $U_<$ are null sets. Therefore, defining $\tilde G_\alpha \equiv \frac{\dKer{G}{\alpha}{1}}{2\bar{\alpha}}\in\RN$ yields
\begin{equation*}
\Del{1}\mathcal{G}(\alpha)(\delta\alpha) = \int_{\RZ} \tilde{G}_\alpha(\bz)~ 2\Re \left(\bar{\alpha}(\bz)~ \delta\alpha(\bz)\right)~ \intd\bz
\end{equation*}
and along with the smoothness assumption, $\mathcal{G}$ has the required form. The converse follows from straightforward computation. \qedhere
\end{proof}

\section{The Hamiltonian Vlasov-Poisson system}
\label{sec:vlasov-poisson}

In this section, we want to use the formal outline described above and apply it to the classical problem of Vlasov-Poisson. The starting point is the full system's energy functional, defined for certain integrable, non-negative distributions on the single particle phase space $f: \RZ=\RX\times\RV\to\RNp$
\begin{equation*}
\HF(f) \equiv \int_{\RZ} \frac{\abs{\bv}^2}{2}~ f(\bz)~ \intd\bz - \frac{1}{2} \int_{\RZ} \int_{\RZ} f(\bz_1)~ f(\bz_2)~ \newtonian{\bx_1-\bx_2}~ \intd\bz_1~ \intd\bz_2,
\end{equation*}
where $\newtonian{\bx}=\frac{\abs{\bx}^{2-d}}{\surS{\ddim}\ddim(2-\ddim)}$ is the Newtonian two-particle interaction potential in $\ddim\geq 3$ dimensions. The first two functional derivatives are now easily seen to be
\begin{align*}
\dKer{H}{f}{1}(\bz) =& \frac{\abs{\bv}^2}{2} - \int_{\RZ} f(\bar{\bz})~ \newtonian{\bx-\bar{\bx}}~ \intd\bar{\bz}, \\
\dKer{H}{f}{2}(\bz_1,\bz_2) =& -\Gamma(\bx_1-\bx_2).
\end{align*}
In fact, the \textbf{Hamiltonian Vlasov-Poisson functional} is upon integration by parts found to be
\begin{align}
\label{eqn:hamilton-vlasov-poisson-functional}
\HVF(\alpha) \equiv&~ \frac 1 2 \Re\int \bar{\alpha}(\bx,\bv)~\bv\cdot \frac{1}{\mi}\naX\alpha(\bx,\bv)~\intd(\bx,\bv)  \\
\nonumber
&- \frac{1}{2} \Re\int \bar{\alpha}(\bx,\bv)~ \left(\int \nabla \newtonian{\bx-\by}~\abs{\alpha(\by,\bw)}^2~\intd(\by,\bw) \right) \cdot \frac{1}{\mi}\naV\alpha(\bx,\bv)~ \intd(\bx,\bv),
\end{align}
while the corresponding \textbf{Hamiltonian Vlasov-Poisson equation} is easily determined to be
\begin{align}
\label{eqn:hamilton-vlasov-poisson-eqn}
\nonumber
\partial_t \alpha(t,\bx,\bv) =& -\bv\cdot\naX\alpha(t,\bx,\bv)+\left(\int\nabla \newtonian{\bx-\by}\abs{\alpha(t,\by,\bw)}^2~\intd(\by,\bw)\right)\cdot\naV\alpha(t,\bx,\bv) \\
&-\alpha(t,\bx,\bv)\int \bar{\alpha}(t,\by,\bw)~\nabla \newtonian{\bx-\by}\cdot\naW\alpha(t,\by,\bw)~\intd(\by,\bw).
\end{align}
The initial condition is given by $\alpha(0,\cdot)=\mathring{\alpha}$ for some function $\mathring{\alpha}:\RX\times\RV\to\CN$. Formulae \eqref{eqn:hamilton-vlasov-poisson-functional} and \eqref{eqn:hamilton-vlasov-poisson-eqn} exactly coincide with the results in \cite{froehlichknowlesschwarz}. We shorten notation by the

\begin{defn}[Characteristic tuple]
\label{defn:characteristic-tuple}
Let $\alpha: \RX\times\RV\to\CN$ be some function. If all the functions
\begin{align*}
\rho(\bx) \equiv& \int_{\RV} \abs{\alpha(\bx,\bv)}^2~\intd\bv, &&F(\bx) \equiv -\convol{\nabla \Gamma}{\rho}(\bx), \\
\varphi(\bx) \equiv& \int_{\RV} \bar{\alpha}(\bx,\bv)~\naV\alpha(\bx,\bv)~\intd\bv, &&K(\bx) \equiv -\convol{\nabla \Gamma}{\varphi}(\bx)
\end{align*}
are well-defined, then $\alpha$ has a \textbf{characteristic tuple} $(\rho,F,\varphi,K)$. In suggestive manner, $\varphi$ is called \textbf{phase density} and $K$ \textbf{phase force}.
\end{defn}

\noindent \eqref{eqn:hamilton-vlasov-poisson-eqn} now reduces to
\begin{equation*}
\partial_t\alpha(t,\bx,\bv) = -\bv \cdot \naX \alpha(t,\bx,\bv) - F(t,\bx) \cdot \naV\alpha(t,\bx,\bv) + K(t,\bx)~ \alpha(t,\bx,\bv).
\end{equation*}

\subsection{Local existence and well-posedness}

The well-posedness theory will be developed in a Banach space of $\Cp{k}{\bz}$ functions with $\Lp{p}{\bz}$ integrable local supremum of all derivatives up to order $k$, as given in the next Definition. For more properties we refer to Appendix \ref{chap:integrable-local-supremum}.

\begin{defn}
\label{defn:loc-sup-Lp}
Let $f:\RN^d\to\RN$ be some measurable function. For any $p\geq1$ and $\kappa\geq d$ Consider the norm
\begin{equation}
\label{eqn:Ap}
\Apn{f}{\kappa,p}{} \equiv \sup_{R\geq 0}~ (1+R)^{-\frac \kappa p} \left(\int_{\RN^d} \sup_{\abs{\bar{\bz}}\leq R} \abs{f(\bz+\bar{\bz})}^p~\intd\bz\right)^{\frac 1p}.
\end{equation}
$\Ap{\kappa,p}{}$ is the set of measurable functions where this norm is finite. In addition, we want to define
\begin{equation}
\label{eqn:Bp}
\Bp{k,\kappa,p}{} \equiv \left\{f\in\Cm{k}{\RN^d\to\CN}: \forall l\leq k, \abs{\alpha}=l: \Del{\alpha}f \in \Ap{\kappa,p}{}\right\}
\end{equation}
with the norm
\begin{equation*}
\Bpn{f}{k,\kappa,p}{} \equiv \left(\sum_{\abs{\alpha}\leq k} \Apn{\Del{\alpha}f}{\kappa,p}{}^p\right)^{\frac 1p}.
\end{equation*}
\end{defn}

\noindent Throughout the section, $\kappa\geq\dim\RZ = 2\ddim$ is a fixed parameter. The central result of this section is

\begin{thm}[Local existence theorem]
\label{thm:local-existence}
Let $\mathring{\alpha}\in\Bp{1,\kappa,2}{\bz}$ be an initial datum. Then there is a positive time of existence $T=T\left(\Bpn{\mathring\alpha}{1,\kappa,2}{\bz}\right)>0$, such that the Hamiltonian Vlasov-Poisson equation gives rise to a unique solution on $[0,T)$.
\end{thm}
\begin{proof}
See Proposition \ref{prop:time-evolution-operator}. \qedhere
\end{proof}

\begin{rmk}
\textbf{(i).} If for some $\gamma>\ddim$ we have $\sup_{\bz} (1+\abs\bz)^\gamma(\abs{\mathring{\alpha}(\bz)}+\abs{\naZ\mathring{\alpha}(\bz)}) < \infty$, then $\mathring{\alpha}\in\Bp{1,\kappa,2}{\bz}$. Therefore, Theorem \ref{thm:local-existence} applies to a wide class of appropriately decaying data. \\

\noindent\textbf{(ii).} For $\abs{\bz}\to\infty$ we always have $\abs{\mathring{\alpha}(\bz)}\to 0$. This is indeed the case because an unbounded sequence of $(\bz_n)$ with $\abs{\alpha(\bz_n)}\geq \epsilon>0$ would imply an infinite norm. Nevertheless, the explicit decay rate is not prescribed. In particular, there need not exist any $\gamma>0$, s.t. $\sup_\bz\abs{\alpha(\bz)} (1+\abs{\bz})^\gamma < \infty$.

Similarly $\abs{\naZ\mathring{\alpha}(\bz)}\to 0$ for $\abs{\bz}\to\infty$. This forbids oscillations at infinity, even if they have decreasing amplitude. At extremes of the phase space, matter has to be locally \textit{almost} equidistributed. Again, no specific decay rate is implied.\\

\noindent\textbf{(iii).} The phase information leads to a peculiar phase space foliation not observable in the Vlasov picture $f=\abs{\alpha}^2$ and somehow similar to Aharonov-Bohm. For example, if $\{\mathring{\alpha}=0\}$ contains a solid tube like $\mathbb{D}^2\times\RN^{2\ddim-2} = \{\bz: \abs{z_1}^2 + \abs{z_2}^2 \leq 1\}$, it is possible to modify $\mathring{\alpha}$ with phase oscillation around this tube in an obvious manner. This associates some kind of winding number to $\mathring{\alpha}$ that will remain constant under time evolution.
\end{rmk}

\noindent The remainder of the section will be devoted to the proof of Theorem \ref{thm:local-existence}. At first, we specify the class of solutions we are interested in.

\begin{defn}[Solution]
\label{defn:solution}
Let $T>0$ and $\alpha\in\Cm{1}{[0,T)\times\RX\times\RV;\CN}$ be some function. It is called \textbf{admissible} if and only if
\begin{enumerate}[label=(\roman*)]
\item for every $t\in[0,T)$, the characteristic tuple $(\rho(t),F(t),\varphi(t),K(t))$ of $\alpha(t)$ is well-defined,
\item $F,\naX F,K,\naX K\in\Cp{0}{(t,\bx)}$ exist and are continuous as functions of time and space, and
\item $\sup_{\tau\leq t} \Lpn{F(\tau)}{\infty}{\bx}<\infty$ for every $t\in[0,T)$.
\end{enumerate}
An admissible function is called a \textbf{local solution} if it satisfies the Hamiltonian Vlasov-Poisson equation. The solution is called \textbf{global} in the case of $T=\infty$.
\end{defn}

\noindent In preparation of our proof, we need three technical lemmata. We introduce the notation 
\begin{equation*}
\minR(a,b) \equiv \inf_{R>0} \frac{(1+R)^a}{\volS{b}R^b} = \frac{a^a (a-b)^{b-a}}{\volS{b} b^b},
\end{equation*}
$\volS{b}$ the Lebesgue volume of the $b$ dimensional unit ball.

\begin{lem}
Let $\alpha\in\Bp{1,\kappa,2}{\bz}$ be some function. For the characteristic tuple holds
\begin{equation*}
(\rho,F,\varphi,K) \in \left(\Wp{1,1}{\bx}\cap\Cp{1}{\bx}\right) \times \Cp{1,1}{\bx} \times \left(\Wp{1,1}{\bx}\cap\Cp{1}{\bx}\right) \times \Cp{1,1}{\bx}.
\end{equation*}
\end{lem}
\begin{proof}
\textbf{(i) Density $\rho$.} Because $\alpha\in\Bp{1,\kappa,2}{\bz}\subseteq\Ap{\kappa,2}{\bz}\subseteq\Lp{2}{\bz}$, $\rho\in\Lp{1}{\bx}$. On the other hand, because $\alpha\in\Ap{\kappa,2}{\bz}$, for any $R>0$ and $\bx$
\begin{align*}
\abs{\rho(\bx)} 
\leq& \frac{1}{\volS{\ddim}R^\ddim} \int_{\ball{\bx}{R}\subset\RX} \sup_{\abs{\bar{\bx}}\leq R} \abs{\rho(\tilde{\bx}+\bar{\bx})}\intd\tilde{\bx} \\
\leq& \frac{1}{\volS{\ddim}R^\ddim} \int_{\RX}\int_{\RV} \sup_{\abs{(\bar{\bx},\bar{\bv})}\leq R} \abs{\alpha(\tilde{\bx}+\bar{\bx},\tilde{\bv}+\bar{\bv})}^2 \intd(\tilde{\bx},\tilde{\bv}) \\
\leq& \frac{(1+R)^{\kappa}}{\volS{\ddim}R^\ddim} \Apn{\alpha}{\kappa,2}{\bz}^2 
\stackrel{R\text{ opt.}}{=} \minR(\kappa,\ddim) \Apn{\alpha}{\kappa,2}{\bz}^2,
\end{align*}
implying $\rho\in\Lp{\infty}{\bx}$. Because $\naX\alpha\in\Ap{\kappa,2}{\bz}\subseteq\Lp{2}{\bz}$, the natural candidate for the derivative is
\begin{equation*}
\naX \rho(\bx) = 2\Re \int_{\RV} \bar{\alpha}(\bx,\bv)~\naX\alpha(\bx,\bv)~\intd\bv.
\end{equation*}
It is indeed valid, because for any $R>0$ and $\bx$
\begin{align*}
&\int_{\RV} \sup_{\abs{\bar{\bx}}\leq R}\abs{\bar{\alpha}(\bx+\bar{\bx},\bv)}~\abs{\naX\alpha(\bx+\bar{\bx},\bv)}\intd\bv \\
\leq& \frac{1}{\volS{\ddim}R^\ddim}\int_{\RZ} \sup_{\abs{\bar{\bz}}\leq 2R} \abs{\bar{\alpha}(\tilde{\bz}+\bar{\bz})}~\abs{\naX\alpha(\tilde{\bz}+\bar{\bz})}~\intd\tilde{\bz} \\
\leq& \frac{\left(1+2R\right)^{\kappa}} {\volS{\ddim}R^\ddim} \Apn{\abs{\bar{\alpha}}\abs{\naX\alpha}}{\kappa,1}{\bz} 
\stackrel{\text{Hölder}}{\leq} \frac{\left(1+2R\right)^{\kappa}} {\volS{\ddim}R^\ddim} \Apn{\alpha}{\kappa,2}{\bz} \Apn{\naX\alpha}{\kappa,2}{\bz} \\
\stackrel{R\text{ opt.}}{=}& 2^\ddim \minR(\kappa,\ddim) \Apn{\alpha}{\kappa,2}{\bz} \Apn{\naX\alpha}{\kappa,2}{\bz}.
\end{align*}
Hence, $\rho\in\Cp{1}{\bx}$. Finally $\naX\rho\in\Lp{1}{\bx}$ as $\alpha,\naX\alpha\in\Lp{2}{\bz}$, yielding also $\rho\in\Wp{1,1}{\bx}$. \\

\noindent\textbf{(ii) Force $F$.} By (i), the density $\rho$ is in $\Lp{1}{\bx}\cap\Cp{1}{\bx}$. By Lemma \ref{lem:newtonian-ineq}, the force $F=-\convol{\nabla \newtonian{\cdot}}{\rho}$ then is in $\Cp{1,1}{\bx}$. \\

\noindent\textbf{(iii) Phase density $\varphi$.} Because $\alpha,\naV\alpha\in\Lp{2}{\bz}$, $\varphi\in\Lp{1}{\bx}$. In addition, $\alpha,\naV\alpha\in\Ap{\kappa,2}{\bz}$ implies $\varphi\in\Lp{\infty}{\bx}$. If $\alpha\in\Ccm{\infty}{\RZ}$, the derivative is seen to be
\begin{equation*}
\naX\varphi(\bx) = 2\mi\Im\int_{\RV} \naX\bar{\alpha}(\bx,\bv)~\naV\alpha(\bx,\bv)~\intd\bv,
\end{equation*}
with similar computations as in (i). This relation can be extended by a density argument to $\alpha\in\Bp{1,\kappa,2}{\bz}$. By identical arguments as in (i), we find $\varphi\in\Wp{1,1}{\bx}\cap\Cp{1}{\bx}$. \\

\noindent\textbf{(iv) Phase force $K$.} As $\varphi$ is in the same space as $\rho$, all arguments can be copied from (ii). \qedhere
\end{proof}

\begin{lem}
\label{lem:local-lipschitz}
Let $\alpha_1,\alpha_2\in\Bp{1,\kappa,2}{\bz}$ be two functions. If $(\rho_i,F_i,\varphi_i,K_i)$ denote the characteristic tuples, then the following inequalities hold:
\begin{enumerate}[label=(\roman*)]
\item $\Lpn{\rho_1-\rho_2}{\infty}{\bx} \leq A(\kappa,d) \left(\Apn{\alpha_1}{\kappa,2}{\bz}+\Apn{\alpha_2}{\kappa,2}{\bz}\right) \Apn{\alpha_1-\alpha_2}{\kappa,2}{\bz}$,
\item $\Lpn{\rho_1-\rho_2}{1}{\bx} \leq \left(\Apn{\alpha_1}{\kappa,2}{\bz} + \Apn{\alpha_2}{\kappa,2}{\bz}\right) \Apn{\alpha_1-\alpha_2}{\kappa,2}{\bz}$,
\item $\Lpn{F_1-F_2}{\infty}{\bx} \leq \cstNewtonLp{1} \minR(\kappa,\ddim)^{1-\frac 1\ddim} \left(\Apn{\alpha_1}{\kappa,2}{\bz} + \Apn{\alpha_2}{\kappa,2}{\bz}\right) \Apn{\alpha_1-\alpha_2}{\kappa,2}{\bz}$,
\item $\Lpn{\varphi_1-\varphi_2}{\infty}{\bx} \leq \minR\left(\kappa,\ddim\right) \left(\Apn{\naV\alpha_1}{\kappa,2}{\bz} + \Apn{\naV\alpha_2}{\kappa,2}{\bz}\right) \Apn{\alpha_1-\alpha_2}{\kappa,2}{\bz}$,
\item $\Lpn{\varphi_1-\varphi_2}{1}{\bx} \leq \left(\Apn{\naV\alpha_1}{\kappa,2}{\bz} + \Apn{\naV\alpha_2}{\kappa,2}{\bz}\right) \Apn{\alpha_1-\alpha_2}{\kappa,2}{\bz}$,
and
\item $\Lpn{K_1-K_2}{\infty}{\bx} \leq \cstNewtonLp{1} \minR\left(\kappa,\ddim\right)^{1-\frac 1\ddim} \left(\Apn{\naV\alpha_1}{\kappa,2}{\bz} + \Apn{\naV\alpha_2}{\kappa,2}{\bz}\right) \Apn{\alpha_1-\alpha_2}{\kappa,2}{\bz}$.
\end{enumerate}
\end{lem}

\begin{proof}
\textbf{(i), (ii), \& (iii).} For any $\bx$, $R>0$,
\begin{align*}
\abs{\rho_1(\bx)-\rho_2(\bx)} 
\leq& \frac{1}{\volS{\ddim}R^\ddim} \int_{\RZ} \sup_{\abs{\bar{\bz}}\leq R} \left(\abs{\alpha_1(\tilde{\bz}+\bar{\bz})} + \abs{\alpha_1(\tilde{\bz}+\bar{\bz})}\right)\abs{\alpha_1(\tilde{\bz}+\bar{\bz}) - \alpha_2(\tilde{\bz}+\bar{\bz})}\intd\tilde{\bz} \\
\stackrel{R~\text{opt.}}{\leq}& A(\kappa,d) \left(\Apn{\alpha_1}{\kappa,2}{\bz}+\Apn{\alpha_2}{\kappa,2}{\bz}\right) \Apn{\alpha_1-\alpha_2}{\kappa,2}{\bz},
\end{align*}
and upon integration,
\begin{align*}
\Lpn{\rho_1-\rho_2}{1}{\bx} \leq& \left(\Lpn{\alpha_1}{2}{\bz} + \Lpn{\alpha_2}{2}{\bz}\right) \Lpn{\alpha_1-\alpha_2}{2}{\bz} \leq \left(\Apn{\alpha_1}{\kappa,2}{\bz} + \Apn{\alpha_2}{\kappa,2}{\bz}\right) \Apn{\alpha_1-\alpha_2}{\kappa,2}{\bz}.
\end{align*}
In combination with inequality \eqref{eqn:DU-estimate}, one finds
\begin{align*}
\Lpn{F_1-F_2}{\infty}{\bx} 
\leq& \cstNewtonLp{1} \Lpn{\rho_1-\rho_2}{1}{\bx}^{\frac 1\ddim} \Lpn{\rho_1-\rho_2}{\infty}{\bx}^{1-\frac 1\ddim} \\
\leq& \cstNewtonLp{1} \minR(\kappa,\ddim)^{1-\frac 1\ddim} \left(\Apn{\alpha_1}{\kappa,2}{\bz} + \Apn{\alpha_2}{\kappa,2}{\bz}\right) \Apn{\alpha_1-\alpha_2}{\kappa,2}{\bz}.
\end{align*} \\

\noindent\textbf{(iv), (v), \& (vi).} For any $\bx$,
\begin{align*}
&\abs{\varphi_1(\bx)-\varphi_2(\bx)} \\
\leq& \abs{\int_{\RV} \left(\bar{\alpha}_1(\bx,\bv)-\bar{\alpha}_2(\bx,\bv)\right) \naV\alpha_1(\bx,\bv) \intd\bv} \\ &+ \abs{\int_{\RV} \bar{\alpha}_2 \left(\naV\alpha_1(\bx,\bv)-\naV\alpha_2(\bx,\bv)\right)\intd\bv} \\
\leq& \int_{\RV} \left(\abs{\naV\alpha_1(\bx,\bv)} + \abs{\naV\alpha_2(\bx,\bv)}\right) \abs{\alpha_1(\bx,\bv) - \alpha_2(\bx,\bv)}\intd\bv \\
\leq& \frac{1}{\volS{\ddim} R^\ddim} \int_{\RZ} \sup_{\abs{\bar{\bz}}\leq R} \left(\abs{\naV\alpha_1(\bz+\bar{\bz})}+\abs{\naV\alpha_2(\bz+\bar{\bz})}\right) \abs{\alpha_1(\bz+\bar{\bz})-\alpha_2(\bz+\bar{\bz})} \intd\bz \\
\stackrel{R~\text{opt.}}{\leq}&~\minR\left(\kappa,\ddim\right) \left(\Apn{\naV\alpha_1}{\kappa,2}{\bz} + \Apn{\naV\alpha_2}{\kappa,2}{\bz}\right) \Apn{\alpha_1-\alpha_2}{\kappa,2}{\bz}, 
\end{align*}
and upon integration,
\begin{align*}
\Lpn{\varphi_1-\varphi_2}{1}{\bx} 
\leq& \left(\Lpn{\naV\alpha_1}{2}{\bz} + \Lpn{\naV\alpha_2}{2}{\bz}\right) \Lpn{\alpha_1-\alpha_2}{2}{\bz} \\
\leq& \left(\Apn{\naV\alpha_1}{\kappa,2}{\bz} + \Apn{\naV\alpha_2}{\kappa,2}{\bz}\right) \Apn{\alpha_1-\alpha_2}{\kappa,2}{\bz}.
\end{align*}
In total, by inequality \eqref{eqn:DU-estimate},
\begin{align*}
\Lpn{K_1-K_2}{\infty}{\bx} 
\leq& \cstNewtonLp{1} \minR\left(\kappa,\ddim\right)^{1-\frac 1\ddim} \left(\Apn{\naV\alpha_1}{\kappa,2}{\bz} + \Apn{\naV\alpha_2}{\kappa,2}{\bz}\right) \Apn{\alpha_1-\alpha_2}{\kappa,2}{\bz}. \quad\qedhere
\end{align*}
\end{proof}

\begin{lem}
\label{lem:symplectomorphism-difference}
Let $Z=(X,V), \bar{Z}=(\bar{X},\bar{V}): \RN\times\RN\times\RX\times\RV\to\RX\times\RV$ the solution maps of the $\Cp{1}{}$-ODE systems
\begin{align*}
\partial_s X(s,t,\bx,\bv) = V(s,t,\bx,\bv),& \quad \partial_s V(s,t,\bx,\bv) = F(s,X(s,t,\bx,\bv)),\quad\text{and} \\
\partial_s \bar{X}(s,t,\bx,\bv) = \bar{V}(s,t,\bx,\bv),& \quad \partial_s \bar{V}(s,t,\bx,\bv) = \bar{F}(s,\bar{X}(s,t,\bx,\bv)),
\end{align*}
where $F, \bar{F}\in \Cp{0}{t}\Cp{1,1}{\bx}$. Then we find for any $s\leq t$
\begin{equation}
\label{eqn:symplectomorphism-difference}
\Lpn{Z(s,t)-\bar{Z}(s,t)}{\infty}{\bz} \leq \int_{s}^{t} \Lpn{F(\tau)-\bar{F}(\tau)}{\infty}{\bx} \exp\left(\int_{s}^{\tau} \left(1+\Lpn{\naX F(\tilde{\tau})}{\infty}{\bx}\right)\intd\tilde{\tau}\right) \intd\tau.
\end{equation}
\end{lem}
\begin{proof}
Let $s\leq t$ and $\bz$ be arbitrary, then we find
\begin{align*}
&\abs{X(s,t,\bz)-\bar{X}(s,t,\bz)} + \abs{V(s,t,\bz)-\bar{V}(s,t,\bz)} \\
\leq& \int_{s}^{t} \left(\abs{V(\tau,t,\bz)-\bar{V}(\tau,t,\bz)}+\abs{F(\tau,X(\tau,t,\bz))-\bar{F}(\tau,\bar{X}(\tau,t,\bz))}\right)~\intd\tau \\
\leq& \int_{s}^{t} \left(\abs{V(\tau,t,\bz)-\bar{V}(\tau,t,\bz)} + \abs{F(\tau,X(\tau,t,\bz))-F(\tau,\bar{X}(\tau,t,\bz))}\right)~\intd\tau \\
&+ \int_{s}^{t} \Lpn{F(\tau)-\bar{F}(\tau)}{\infty}{\bx}~\intd\tau \\
\leq& \int_{s}^{t} \left(1+\Lpn{\naX F(\tau)}{\infty}{\bx}\right)\left(\abs{X(\tau,t,\bz)-\bar{X}(\tau,t,\bz)}+\abs{V(\tau,t,\bz)-\bar{V}(\tau,t,\bz)}\right)~\intd\tau \\
&+ \int_{s}^{t} \Lpn{F(\tau)-\bar{F}(\tau)}{\infty}{\bx}~\intd\tau,
\end{align*}
yielding a bound by the Gronwall Lemma
\begin{align*}
\abs{Z(s,t,\bz)-\bar{Z}(s,t,\bz)} \leq&\abs{X(s,t,\bz)-\bar{X}(s,t,\bz)} + \abs{V(s,t,\bz)-\bar{V}(s,t,\bz)} \\
\leq& \int_{s}^{t} \Lpn{F(\tau)-\bar{F}(\tau)}{\infty}{\bx} \exp\left(\int_{s}^{\tau}\left(1+\Lpn{\naX F(\tilde{\tau})}{\infty}{\bx}\right)~\intd\tilde{\tau}\right)~\intd\tau. \quad \qedhere
\end{align*}
\end{proof}

\begin{lem}[Transport formula]
\label{lem:transport-formula}
Let $\alpha\in\Cm{1}{[0,T)\times\RX\times\RV;\CN}$ be an admissible function in the sense of Definition \ref{defn:solution}. Equivalent statements are
\begin{enumerate}[label=(\roman*)]
\item $\alpha$ is a local solution of the Hamiltonian Vlasov-Poisson equation.
\item For any $(t,\bz)$ we have
\begin{equation}
\label{eqn:transport-formula}
\alpha(t,\bz) = \alpha(0,Z(0,t,\bz))~\exp\left(\int_{0}^{t} K(\tau,X(\tau,t,\bz))~\intd\tau\right),
\end{equation}
where $Z(s,t,\bz)=(X,V)(s,t,\bz)$ shall be the solution map of the characteristic system
\begin{align*}
\del{}{s} X(s,t,\bx,\bv) = V(s,t,\bx,\bv), \del{}{s} V(s,t,\bx,\bv) = F(s,X(s,t,\bx,\bv))
\end{align*}
with initial condition $Z(t,t,\bx,\bv)=(\bx,\bv)$.
\end{enumerate}
\end{lem}
\begin{proof}
As $\alpha$ is admissible, we find $F\in\Cp{0}{t}\Cp{1}{\bx}$ and the solution map $Z$ is $\Cp{1}{}$. The explicit computation can now be copied from the formal proof in Proposition \ref{prop:general-transport-formula} \qedhere
\end{proof}

\begin{lem}[Existence]
\label{lem:local-solution}
Let $M>0$ be some number. Then exists a positive time of existence $T(M)>0$, s.t. every compactly supported initial datum $\mathring\alpha\in\Ccm{1}{\RX\times\RV;\CN}$ with $\Bpn{\mathring\alpha}{1,\kappa,2}{\bz}\leq M$ gives rise to a $\Cp{1}{}$ solution 
\begin{equation*}
\alpha: [0,T(M))\times\RX\times\RV
\end{equation*}
of the Hamiltonian Vlasov-Poisson equation.
\end{lem}
\begin{proof}
\textbf{(i) Iterative scheme.} The solution is found with an iterative scheme and suitable contractions. At first, we define $\alpha_0(t,\bx,\bv)\equiv\mathring\alpha(\bx,\bv)$ on $\RNp\times\RX\times\RV$. If $\alpha_n$ is defined, it has a characteristic tuple $(\rho_n,F_n,\varphi_n,K_n)$ due to its compact support on any compact time interval. $Z_n(s,t,\bx,\bv)=(X_n,V_n)(s,t,\bx,\bv)$ then shall be the solution map of the characteristic system
\begin{align*}
\del{}{s} X_n(s,t,\bx,\bv) = V_n(s,t,\bx,\bv), \del{}{s} V_n(s,t,\bx,\bv) = F_n(s,X_n(s,t,\bx,\bv))
\end{align*}
with initial condition $Z_n(t,t,\bx,\bv)=(\bx,\bv)$. Finally, the iteration step is completed by defining
\begin{align*}
\alpha_{n+1}(t,\bx,\bv) \equiv \mathring\alpha(Z_n(0,t,\bx,\bv))~ \exp\left(\int_{0}^{t}K_n(\tau,X_n(\tau,t,\bx,\bv))~\intd \tau\right).
\end{align*}
\\

\noindent \textbf{(ii) Well-definedness.} Following instructive ideas similar to \cite[pp.396]{rein}, we want to verify that the iteration scheme is well-defined, $\alpha_n\in\Cm 1 {\RNp\times\RX\times\RV}$, and that for any $t$ $\alpha_n(t)\in\Ccm1{\RX\times\RV}$. For the latter claim we control the size of the support. Therefore, we define
\begin{align*}
R_n(t) \equiv \sup \left\{\abs{\bx} : (\bx,\bv)\in\supp~\alpha_n(t)\right\}, 
P_n(t) \equiv \sup \left\{\abs{\bv} : (\bx,\bv)\in\supp~\alpha_n(t)\right\}
\end{align*}
and claim that they both remain bounded on finite time intervals.

The regularity condition is obviously fulfilled for $\alpha_0$ by construction and choice of the initial value. The size of the support is found to be $R_0(t)=\mathring R, P_0(t)=\mathring P>0$ for the minimal numbers, s.t. $\supp~\mathring\alpha \subseteq \overline{\ball{\bn}{\mathring R}}\times \overline{\ball{\bn}{\mathring P}}\subseteq\RX\times\RV$.

Now assume, the claims to hold for some $n\geq0$. From the iterative scheme it is obvious, that $\rho_n(t)$ has compact support and is continuously differentiable, $F_n(t)$ is continuously differentiable and bounded on $\RX$. $\varphi_n$ is continuously differentiable, because for $\mathcal{C}^2$ functions $\alpha_n(t,\bx,\bv)$ using integration by parts, for any first order differentiation $\partial$ w.r.t. $(t,\bx)$
\begin{align*}
\partial\phi_n(t,\bx) = \int_{\RV} \left(\partial\bar{\alpha}_n(t,\bx,\bv)~\naV\alpha_n(t,\bx,\bv) - \naV\bar{\alpha}_n(t,\bx,\bv)~\partial\alpha_n(t,\bx,\bv)\right)\intd\bv,
\end{align*}
an identity conserved by standard approximation arguments. At last, also $K_n\in\Cp{0}{t}\Cp{1}{\bx}$.

By standard theory on ordinary differential equations, the solution map $Z_n$ then is $\mathcal{C}^1$. Therefore, $\alpha_{n+1}\in \Cm{1}{\RNp\times\RX\times\RV}$ and because $K_n(t,\bx)$ is imaginary, 
\begin{align}
\label{eqn:abs-preservation}
\abs{\alpha(t,\bx,\bv)} = \abs{\mathring\alpha(Z_n(0,t,\bx,\bv))}.
\end{align}
Hence, as $Z_n(s,t,\cdot)$ is a diffeomorphism of $\RX\times\RV$, 
\begin{align}
\nonumber
P_{n+1}(t) =& \sup\left\{\abs\bv : (\bx,\bv)\in\supp~\alpha_n(t)\right\} = \sup\left\{\abs{V_n(t,0,\bx,\bv)} : (\bx,\bv)\in\supp~\mathring\alpha \right\} \\
\nonumber
\leq& \sup \left\{\abs{V_n(0,0,\bx,\bv)} + \int_0^t \abs{F_n(s,X_n(s,0,\bx,\bv))}~\intd s : (\bx,\bv)\in\supp~\mathring\alpha\right\} \\
\nonumber
\leq& \sup \left\{\abs{\bv} + \int_{0}^{t} \Lpn{F_n(s)}{\infty}{\bx}~\intd s : (\bx,\bv)\in\supp~\mathring\alpha\right\} \\
\nonumber
\stackrel{\text{Lem.\ref{lem:newtonian-ineq}}-\eqref{eqn:DU-estimate}}{\leq}&~ \mathring P + \cstNewtonLp{1} \Lpn{\mathring\alpha}{2}{\bz}^{\frac 2\ddim}  \left(\volS{\ddim} \Lpn{\mathring\alpha}{\infty}{\bz}^2\right)^{1-\frac 1\ddim} \int_0^t P_n(s)^{\ddim-1}~\intd s,
\end{align}
and
\begin{align*}
R_{n+1}(t) =& \sup \left\{\abs{\bx} : (\bx,\bv)\in\supp~\alpha_{n+1}(t)\right\} = \sup \left\{\abs{X_n(t,0,\bx,\bv)} : (\bx,\bv)\in\supp~\mathring\alpha\right\} \\
\leq& \sup \left\{\abs{X_n(0,0,\bx,\bv)}+\int_{0}^{t} \abs{V_n(s,0,\bx,\bv)}~\intd s : (\bx,\bv)\in\supp~\mathring\alpha \right\} \\
\leq&~\mathring R + \int_{0}^{t} P_n(s)~\intd s,
\end{align*}
yielding the second claim. The iterative scheme is well-defined and all objects exist. \\

\noindent \textbf{(iii) Preserved $\Lp{p}{}$ norms.} As $Z_n(s,t,\cdot)$ is a symplectomorphism of $\RX\times\RV$, it preserves the Lebesgue measure. By relation \eqref{eqn:abs-preservation}, one finds
\begin{align}
\label{eqn:lp-preservation}
\Lpn{\alpha_n(t)}{p}{\bz} \stackrel{\eqref{eqn:abs-preservation}}{=} \left(\int_{\RX\times\RV}\abs{\mathring\alpha(Z_{n-1}(0,t,\bz))}^p~\intd\bz\right)^{\frac 1p} = \Lpn{\mathring\alpha}{p}{\bz}
\end{align}
and
\begin{align*}
\Lpn{\rho_n(t)}{1}{\bx} = \int_{\RX}\left(\int_{\RV} \abs{\alpha_n(t,\bx,\bv)}^2~\intd\bv\right)~\intd\bx \stackrel{\eqref{eqn:lp-preservation}}{=} \Lpn{\mathring\alpha}{2}{\bz}^2.
\end{align*}

\noindent \textbf{(iv) Time of existence.} We want to prove a contraction scheme in the space $\Lp{\infty}{t}\Ap{\kappa,2}{\bz}$. Therefore, it is necessary to show that all convergents $\alpha_n(t)$ are uniformly bounded in this norm. 

At first, we need to compute the difference of the flow $Z_n$ generated under the external potential of $\alpha_n(t)$ and the free flow $\bar{Z}=(\bar{X},\bar{V})$, solving the ODE system
\begin{equation*}
\del{}s \bar{X}(s,t,\bx,\bv) = \bar{V}(s,t,\bx,\bv), \quad \del{}s \bar{V}(s,t,\bx,\bv) = 0.
\end{equation*}
By Lemma \ref{lem:symplectomorphism-difference}, 
\begin{equation}
\label{eqn:proof-1}
\Lpn{Z_n(s,t)-\bar{Z}(s,t)}{\infty}{\bz} \leq \int_{s}^{t} \Lpn{F_n(\tau)}{\infty}{\bx} e^{\tau-s}\intd\tau.
\end{equation}
This implies for any $t\geq0$
\begin{align}
\nonumber &\Apn{\alpha_{n+1}(t)}{\kappa,2}{\bz} 
= \sup_{R\geq 0} (1+R)^{-\frac{\kappa}{2}} \left(\int_{\RZ} \sup_{\abs{\bar{\bz}}\leq R} \abs{\alpha_{n+1}(t,\bz+\bar{\bz})}^2\intd\bz\right)^{\frac 12} \\
\nonumber =& \sup_{R\geq 0} (1+R)^{-\frac{\kappa}{2}} \left(\int_{\RZ} \sup_{\abs{\bar{\bz}}\leq R} \abs{\mathring\alpha(Z_n(0,t,\bz+\bar{\bz})}^2\intd\bz\right)^{\frac 12} \\
\nonumber \leq& \sup_{R\geq 0} (1+R)^{-\frac{\kappa}{2}} \left(\int_{\RZ} \sup_{\abs{\bar{\bz}}\leq R} \sup_{\abs{\bar{\bar{\bz}}}\leq \Lpn{Z_n(0,t)-\bar{Z}(0,t)}{\infty}{\bz}} \abs{\mathring\alpha(\bar{Z}(0,t,\bz+\bar{\bz})+\bar{\bar{\bz}})}^2\intd\bz\right)^{\frac 12} \\
\nonumber \leq& \sup_{R\geq 0} (1+R)^{-\frac{\kappa}{2}} \left(\int_{\RZ} \sup_{\abs{\bar{\bz}}\leq R} \sup_{\abs{\bar{\bar{\bz}}}\leq \int_{0}^{t} \Lpn{F_n(\tau)}{\infty}{\bx} e^{\tau} \intd\tau} \abs{\mathring\alpha(\bx-t\bv + \bar{\bx}-t\bar{\bv}+\bar{\bar{\bx}},\bv+\bar{\bv}+\bar{\bar{\bv}})}^2\intd\bz\right)^{\frac 12} \\
\nonumber \stackrel{(*)}{\leq}& \sup_{R\geq 0} (1+R)^{-\frac{\kappa}{2}} \left(\int_{\RZ} \sup_{\abs{\bar{\bz}}\leq (1+t)R + \int_{0}^{t} \Lpn{F_n(\tau)}{\infty}{\bx} e^{\tau} \intd\tau} \abs{\mathring\alpha(\bz+\bar{\bz})}^2\intd\bz\right)^{\frac 12} \\
\nonumber \leq& \sup_{R\geq 0} \left(\frac{(1+t)R + \int_{0}^{t} \Lpn{F_n(\tau)}{\infty}{\bx} e^{\tau} \intd\tau}{1+R}\right)^{\frac{\kappa}{2}} \Apn{\mathring\alpha}{\kappa,2}{\bz} \\
\label{eqn:key-bound} \leq& \left(1+t+\int_{0}^{t} \Lpn{F_n(\tau)}{\infty}{\bx} e^{\tau} \intd\tau\right)^{\frac{\kappa}{2}} \Apn{\mathring\alpha}{\kappa,2}{\bz} \\
\nonumber \leq& \left(1+t+\cstNewtonLp{1}\minR(\kappa,\ddim)^{1-\frac{1}{\ddim}} \int_{0}^{t} \Apn{\alpha_n(\tau)}{\kappa,2}{\bz}^2 e^\tau\intd\tau\right)^{\frac{\kappa}{2}} \Apn{\mathring\alpha}{\kappa,2}{\bz},
\end{align}
where at $(*)$ the volume preserving substitution $(\bx,\bv)\mapsto(\bx+t\bv,\bv)$ is applied. Inspired by this integral inequality, we infer that $\bound{\alpha}{\kappa,2}(t)\equiv\sup_n\Apn{\alpha_n(t)}{\kappa,2}{\bz}$ is uniformly bounded by the solution $\globalBound{\alpha}{\kappa,2}(t)$ of the integral equation
\begin{equation*}
\globalBound{\alpha}{\kappa,2}(t) = M \left(1+t+\cstNewtonLp{1} \minR(\kappa,\ddim)^{1-\frac{1}{\ddim}} \int_{0}^{t} \globalBound{\alpha}{\kappa,2}(\tau)^2 e^\tau\intd\tau\right)^{\frac{\kappa}{2}}.
\end{equation*}
The maximal solution of this integral inequality exists on some finite time interval defining $T=T(M)$. We will call this interval $I=I(M)=[0,T(M))$. We remark that all bounds further derived in this proof can be ultimately bounded by some expression in $\globalBound{\alpha}{\kappa,2}$. Therefore, they only depend on the bound $M$ of the initial datum $\mathring\alpha$ and the fundamental parameter $\kappa$. \\

\noindent \textbf{(v) Uniform bounds and Lipschitz bounds on density and force.} For technical reasons, we need to prove various bounds. A starting point is uniform global bounds for $\rho$ and $F$, namely
\begin{align*}
\bound{\rho}{\infty}(t) \equiv& \sup_n \Lpn{\rho_n(t)}{\infty}{\bx} \leq \minR(\kappa,\ddim) \sup_n \Apn{\alpha_n(t)}{\kappa,2}{\bz}^2 \leq \minR(\kappa,\ddim) \globalBound{\alpha}{\kappa,2}(t)^2 \equiv \globalBound{\rho}{\infty}(t), \\
\bound{\rho}{1}(t) \equiv& \sup_n \Lpn{\rho_n(t)}{1}{\bx} \stackrel{Z_{n-1}\text{ sympl.}}{=} \Lpn{\mathring\alpha}{2}{\bz}^2 \leq \Apn{\mathring\alpha}{\kappa,2}{\bz}^2 \leq M^2 \equiv \globalBound{\rho}{1}(t),
\end{align*}
and
\begin{align*}
\bound{F}{\infty}(t) \equiv \sup_n \Lpn{F_n(t)}{\infty}{\bx} \leq \cstNewtonLp{1} \minR(\kappa,\ddim)^{1-\frac 1\ddim} \bound{\alpha}{\kappa,2}(t)^2 \leq \cstNewtonLp{1} \minR(\kappa,\ddim)^{1-\frac 1\ddim} \globalBound{\alpha}{\kappa,2}(t)^2 \equiv \globalBound{F}{\infty}(t).
\end{align*}
Upon applying \eqref{eqn:proof-1}, we find for any $s\leq t$
\begin{align*}
\bound{Z-\bar{Z}}{\infty}(s,t) \equiv& \sup_n \Lpn{Z_n(s,t)-\bar{Z}(s,t)}{\infty}{\bz} \stackrel{\eqref{eqn:proof-1}}{\leq} \int_{s}^{t} \bound{F}{\infty}(\tau) e^{\tau-s} \intd\tau \\
\leq& \int_{s}^{t} \globalBound{F}{\infty}(\tau) e^{\tau-s}\intd\tau \equiv \globalBound{Z-\bar{Z}}{\infty}(s,t).
\end{align*}
Similarly, we want to prove bounds for
\begin{align*}
\bound{\nabla\rho}{\infty}(t) \equiv \sup_n \Lpn{\naX\rho_n(t)}{\infty}{\bx}
\quad\text{and}\quad
\bound{\nabla F}{\infty}(t) \equiv \sup_n \Lpn{\naX F_n(t)}{\infty}{\bx}.
\end{align*}
We have for $R>0$
\begin{align*}
&\abs{\naX\rho_{n+1}(t,\bx)} \\
\leq& \abs{\naX\Lpn{\alpha_{n+1}(t,\bx,\bv)}{2}{\bv}^2} 
= \abs{\naX\Lpn{\mathring\alpha(Z_n(0,t,\bx,\bv))}{2}{\bv}^2} \\
\leq& \int_\RV \abs{(\naZ\abs{\mathring\alpha}^2)(Z_n(0,t,\bx,\bv))}~\abs{\naX Z_n(0,t,\bx,\bv)}~\intd\bv \\
\leq& \frac {1}{\volS{\ddim} R^d} \int_{\RZ} \sup_{\abs{\bar{\bz}}\leq R} \abs{(\naZ\abs{\mathring\alpha}^2)(Z_n(0,t,\bz+\bar{\bz}))} \intd\bz~\Lpn{\naX Z_n(0,t)}{\infty}{\bz} \\
\leq& \frac {1}{\volS{\ddim} R^d} \int_{\RZ} \sup_{\abs{\bar{\bz}}\leq (1+t)R} \sup_{\abs{\bar{\bar{\bz}}}\leq \bound{Z-\bar{Z}}{\infty}(0,t)} \abs{(\naZ\abs{\mathring\alpha}^2)(\bar{Z}(0,t,\bz)+\bar{\bz}+\bar{\bar{\bz}})} \intd\bz~\Lpn{\naX Z_n(0,t)}{\infty}{\bz} \\
=& \frac {1}{\volS{\ddim} R^d} \int_{\RZ} \sup_{\abs{\bar{\bz}}\leq (1+t)R+\bound{Z-\bar{Z}}{\infty}(0,t)} \abs{(\naZ\abs{\mathring\alpha}^2)(\bar{Z}(0,t,\bz)+\bar{\bz})} \intd\bz~\Lpn{\naX Z_n(0,t)}{\infty}{\bz} \\
\leq&~ 2\frac {1}{\volS{\ddim} R^d}~ \left((1+t)R+\bound{Z-\bar{Z}}{\infty}(0,t)\right)^{\kappa}~ \Apn{\naZ\mathring\alpha}{\kappa,2}{\bz} \Apn{\mathring\alpha}{\kappa,2}{\bz} \Lpn{\naX Z_n(0,t)}{\infty}{\bz} \\
\stackrel{R\text{ opt.}}{\leq}&~ 2 \minR\left(\kappa,\ddim\right)~ \left(1+t+\bound{Z-\bar{Z}}{\infty}(0,t)\right)^{\kappa}~ M^2 \Lpn{\naX Z_n(0,t)}{\infty}{\bz}
\end{align*}
and, estimating the equations of variation of $Z_n=(X_n,V_n)$,
\begin{align*}
\del{}{s}\naX X_n(s,t,\bz) = \naX V_n(s,t,\bz), \quad \del{}{s} \naX V_n(s,t,\bz) = \naX F_n(s,X_n(s,t,\bz))\cdot \naX X_n(s,t,\bz),
\end{align*}
we find for any $\bz=(\bx,\bv)\in\RX\times\RV$, $0\leq s\leq t$,
\begin{align*}
&&&\abs{\naX Z_n(s,t,\bz)} \leq \abs{\naX X_n(s,t,\bz)} + \abs{\naX V_n(s,t,\bz)} \\ 
&\leq&& \abs{\id_{\RX}} + \int_{s}^{t} \left(\abs{\naX V_n(\tau,t,\bz)}+\abs{\naX F_n(\tau,X_n(\tau,t,\bz))\cdot\naX X_n(\tau,t,\bz)}\right)~\intd\tau \\
&\leq&&~1 + \int_{s}^{t} \left(1+\Lpn{\naX F_n(\tau)}{\infty}{\bx}\right)~\left(\abs{\naX X_n(\tau,t,\bz)}+\abs{\naX V_n(\tau,t,\bz)}\right)~\intd\tau \\
\stackrel{\text{Gronwall}}{\Rightarrow} &&& \Lpn{\naX Z_n(s,t)}{\infty}{\bz} \leq \exp\left(\int_{s}^{t} \left(1+\Lpn{\naX F_n(\tau)}{\infty}{\bx}\right)~\intd\tau\right),
\end{align*}
and, by an identical argument,
\begin{align*}
\Lpn{\naV Z_n(s,t)}{\infty}{\bz} \leq \exp\left(\int_{s}^{t} \left(1+\Lpn{\naX F_n(\tau)}{\infty}{\bx}\right)~\intd\tau\right).
\end{align*}
Combining the estimates, one has
\begin{align}
\label{eqn:nax-rho-bound}
\Lpn{\naX\rho_{n+1}(t)}{\infty}{\bx} \leq 
\underbrace{2 \minR\left(\kappa,\ddim\right)~ \left(1+t+\bound{Z-\bar{Z}}{\infty}(0,t)\right)^{\kappa}~ M^2}_{\leq 2\minR(\kappa,\ddim) \left(1+t+\globalBound{Z-\bar{Z}}{\infty}(0,t)\right)^\kappa M^2 \equiv H_{M,\kappa}(t)} \exp\left(\int_{0}^{t} \left(1+\Lpn{\naX F_n(\tau)}{\infty}{\bx}\right)~\intd\tau\right).
\end{align}
Application of Lemma \ref{lem:newtonian-ineq} yields
\begin{align*}
&\Lpn{\naX F_{n+1}(t)}{\infty}{\bx} \leq \cstNewtonLn \left[\left(1+\Lpn{\rho_{n+1}(t)}{\infty}{\bx}\right) \left(1+\ln_+\Lpn{\naX\rho_{n+1}(t)}{\infty}{\bx}\right) + \Lpn{\rho_{n+1}(t)}{1}{\bx}\right] \\
\leq& \cstNewtonLn \left[\left(1 + \globalBound{\rho}{\infty}(t)\right) \left(1+\ln_+ H_{M,\kappa}(t) + t+ \int_{0}^{t}\Lpn{\naX F_n(\tau)}{\infty}{\bx}~\intd\tau\right) + \globalBound{\rho}{1}(t)\right].
\end{align*}
By an inductive Gronwall argument, we can find a bound
\begin{align*}
\bound{\nabla F}{\infty}(t) \equiv \sup_n \Lpn{\naX F_n(t)}{\infty}{\bx} \leq& \cstNewtonLn \left[\left(1 + \globalBound{\rho}{\infty}(t)\right) \left(1+\ln_+ H_{M,\kappa}(t) + t\right) + \globalBound{\rho}{1}\right] \\
&\cdot \exp\left(\cstNewtonLn \left(1 + \globalBound{\rho}{\infty}(t)\right)t\right) \equiv \globalBound{\nabla F}{\infty}(t).
\end{align*}
At last, $\bound{\nabla F}{\infty}$ and estimate \eqref{eqn:nax-rho-bound} prove the finiteness of $\bound{\nabla\rho}{\infty}$, i.e.,
\begin{align*}
\bound{\nabla\rho}{\infty}(t) \leq& 2\minR(\kappa,\ddim) \left(1+t+\globalBound{Z-\bar{Z}}{\infty}(0,t)\right)^\kappa M^2 \equiv \globalBound{\nabla\rho}{\infty}(t).
\end{align*}
Finally, we have shown that for any $0\leq s\leq t<T(M)$
\begin{equation}
\bound{\nabla Z}{\infty}(s,t) \equiv \sup_n \Lpn{\naZ Z_n(s,t)}{\infty}{\bz} 
\leq 2 \exp\left(\int_{s}^{t}\left(1+\globalBound{\nabla F}{\infty}(\tau)\right)\intd\tau\right) \equiv \globalBound{\nabla Z}{\infty}(s,t)
\end{equation}
is finite. \\

\noindent \textbf{(vi) Uniform bounds on phase density $\varphi$ and phase force $K$.} We want to prove uniform global bounds $\bound{\varphi}{\infty}, \bound{K}{\infty}$, i.e.,
\begin{align}
\bound{\varphi}{\infty}(t) \equiv \sup_n \Lpn{\varphi_n(t)}{\infty}{\bx} 
\text{~~~and~~~} 
\bound{K}{\infty}(t) \equiv \sup_n \Lpn{K_n(t)}{\infty}{\bx}
\end{align}
remain finite on $I$. We find
\begin{align*}
\varphi_{n+1}(t,\bx) =& \int_{\RV} \bar{\alpha}_{n+1}(t,\bx,\bv)~\naV \alpha_{n+1}(t,\bx,\bv)~\intd\bv \\
=& \int_{\RV} \overline{\mathring\alpha(Z_n(0,t,\bx,\bv))}~ \naV\left(\mathring\alpha(Z_n(0,t,\bx,\bv))\right)~\intd\bv \\
&+ \int_{\RV} \abs{\mathring\alpha(Z_n(0,t,\bx,\bv))}^2~\naV\left(\int_{0}^{t} K_n(s,X_n(s,t,\bx,\bv))~\intd s\right)\intd\bv \\
=& \int_{\RV} \overline{\mathring\alpha(Z_n(0,t,\bx,\bv))}~ \naV\left(\mathring\alpha(Z_n(0,t,\bx,\bv))\right)~\intd\bv \\
&- \int_{\RV} \naV\left(\abs{\mathring\alpha(Z_n(0,t,\bx,\bv))}^2\right)~\left(\int_{0}^{t} K_n(s,X_n(s,t,\bx,\bv))~\intd s\right)\intd\bv.
\end{align*}
Upon estimating the absolute value and multiply using the Hölder inequality, one derives
\begin{align*}
&\abs{\varphi_{n+1}(t,\bx)} 
\leq  \abs{\rho_{n+1}(t,\bx)}^{\frac 12} \Lpn{\naV\left(\mathring\alpha(Z_n(0,t,\bx,\bv))\right)}{2}{\bv} \left(1+2\int_{0}^{t} \Lpn{K_n(\tau)}{\infty}{\bx}~\intd\tau\right) \\
\leq&  \abs{\rho_{n+1}(t,\bx)}^{\frac 12} \Lpn{\left(\naZ\mathring\alpha\right)(Z_n(0,t,\bx,\bv))}{2}{\bv} \Lpn{\naV Z_n(0,t)}{\infty}{\bz} \left(1+2\int_{0}^{t} \Lpn{K_n(\tau)}{\infty}{\bx}~\intd\tau\right) \\
\leq& \bound{\rho}{\infty}(t)^{\frac 12} \left(1+t+\bound{Z-\bar{Z}}{\infty}(0,t)\right)^{\frac {\kappa}2} \minR(\kappa,\ddim)^{\frac 12} \Apn{\naZ\mathring\alpha}{\kappa,2}~\Lpn{\naV Z_n(0,t)}{\infty}{\bz} \left(1+2\int_{0}^{t} \Lpn{K_n(\tau)}{\infty}{\bx}~\intd\tau\right) \\
\leq& \globalBound{\rho}{\infty}(t)^{\frac 12} \left(1+t+\globalBound{Z-\bar{Z}}{\infty}(0,t)\right)^{\frac {\kappa}2} \minR(\kappa,\ddim)^{\frac 12} M~\globalBound{\nabla Z}{\infty}(0,t) \left(1+2\int_{0}^{t} \Lpn{K_n(\tau)}{\infty}{\bx}~\intd\tau\right).
\end{align*}
For the integral norm, one finds
\begin{align*}
&\Lpn{\varphi_{n+1}(t)}{1}{\bx} \\
\leq& \Lpn{\mathring\alpha\circ Z_n(0,t)}{2}{\bz} \Lpn{\left(\naZ\mathring\alpha\right)\circ Z_n(0,t)}{2}{\bz} \Lpn{\naV Z_n(0,t)}{\infty}{\bz} \left(1+2\int_{0}^{t} \Lpn{K_n(\tau)}{\infty}{\bx}~\intd\tau\right) \\
=& \Lpn{\mathring\alpha}{2}{\bz} \Lpn{\naZ\mathring\alpha}{2}{\bz} \Lpn{\naV Z_n(0,t)}{\infty}{\bz} \left(1+2\int_{0}^{t} \Lpn{K_n(\tau)}{\infty}{\bx}~\intd\tau\right) \\
\leq&~ M^2~\globalBound{\nabla Z}{\infty}(0,t) \left(1+2\int_{0}^{t} \Lpn{K_n(\tau)}{\infty}{\bx}~\intd\tau\right).
\end{align*}
As $K_n$ is only a (imaginary valued) convolution, by Lemma \ref{lem:newtonian-ineq},
\begin{align*}
&\Lpn{K_{n+1}(t)}{\infty}{\bx} 
\leq \cstNewtonLp{1} \Lpn{\varphi_{n+1}(t)}{1}{\bx}^{\frac 1\ddim} \Lpn{\varphi_{n+1}(t)}{\infty}{\bx}^{1-\frac 1\ddim} 
\leq \cstNewtonLp{1} M^{1+\frac 1\ddim} \globalBound{\nabla Z}{\infty}(0,t) \\
&\cdot ~\left(\globalBound{\rho}{\infty}(t)^{\frac 12} \left(1+t+\globalBound{Z-\bar{Z}}{\infty}(0,t)\right)^{\frac {\kappa}2} \minR(\kappa,\ddim)^{\frac 12}\right)^{1-\frac 1\ddim}~\left(1+2\int_{0}^{t} \Lpn{K_n(\tau)}{\infty}{\bx}~\intd\tau\right),
\end{align*}
which proves the uniform finiteness of $\bound{K}{\infty}$ on the minimal interval of existence. Now we find by an inductive Gronwall argument,
\begin{align*}
&\sup_n\Lpn{K_{n}(t)}{\infty}{\bx} \\
\leq&~ \cstNewtonLp{1} M^{1+\frac 1\ddim} \globalBound{\nabla Z}{\infty}(0,t)~ \left(\globalBound{\rho}{\infty}(t)^{\frac 12} \left(1+t+\globalBound{Z-\bar{Z}}{\infty}(0,t)\right)^{\frac {\kappa}2} \minR(\kappa,\ddim)^{\frac 12}\right)^{1-\frac 1\ddim} \\
&\cdot \exp\left(2\cstNewtonLp{1} M^{1+\frac 1\ddim} \globalBound{\nabla Z}{\infty}(0,t)
~\left(\globalBound{\rho}{\infty}(t)^{\frac 12} \left(1+t+\globalBound{Z-\bar{Z}}{\infty}(0,t)\right)^{\frac {\kappa}2} \minR(\kappa,\ddim)^{\frac 12}\right)^{1-\frac 1\ddim} t\right) \\
\equiv&~ \globalBound{K}{\infty}(t).
\end{align*}
The estimates on $\varphi_{n+1}$ in turn imply the finiteness of
\begin{align*}
\bound{\varphi}{\infty}(t) \equiv&
\sup_n\Lpn{\varphi_n(t)}{\infty}{\bx} 
\leq \globalBound{\rho}{\infty}(t)^{\frac 12} \left(1+t+\globalBound{Z-\bar{Z}}{\infty}(0,t)\right)^{\frac {\kappa}2} \\ 
&\cdot \minR(\kappa,\ddim)^{\frac 12} M~\globalBound{\nabla Z}{\infty}(0,t) \left(1+2\int_{0}^{t} \globalBound{K}{\infty}(\tau)~\intd\tau\right) \equiv \globalBound{\varphi}{\infty}(t), \\
\bound{\varphi}{1}(t) \equiv& \sup_n \Lpn{\varphi_n(t)}{1}{\bx} \leq M^2~\globalBound{\nabla Z}{\infty}(0,t) \left(1+2\int_{0}^{t} \globalBound{K}{\infty}(\tau)~\intd\tau\right) \equiv \globalBound{\varphi}{1}(t).
\end{align*}
\noindent \textbf{(vii) Uniform Lipschitz bounds on phase density $\varphi$ and phase force $K$.} Finally, we need to show that $\bound{\nabla\varphi}{\infty}(t) \equiv \sup_n\Lpn{\naX\varphi_n(t)}{\infty}{\bx}$ and $\bound{\nabla K}{\infty}(t) \equiv \sup_n\Lpn{\naX K_n(t)}{\infty}{\bx}$ remain finite. Indeed, we compute the matrix valued derivative of $\varphi_{n+1}$
\begin{align*}
\naX\varphi_{n+1}(t,\bx) =&~ 2\mi~\Im \int_{\RV} \naX\bar{\alpha}_{n+1}(t,\bx,\bv)\cdot\naV\alpha_{n+1}(t,\bx,\bv)~\intd\bv \\
=&~2\mi~\Im \int_\RV \left(\naX (\bar{\mathring\alpha}(Z_n(0,t,\bx,\bv))) + \bar{\mathring\alpha}(Z_n(0,t,\bx,\bv)) \phantom{\int} \right. \\ 
&\left.\cdot\int_{0}^{t}(\naX \bar{K_n})(\tau,X_n(\tau,t,\bx,\bv))\cdot (\naX X_n)(\tau,t,\bx,\bv)~\intd\tau\right) \\
&\cdot\left(\naV (\mathring\alpha(Z_n(0,t,\bx,\bv)))
+ \mathring\alpha(Z_n(0,t,\bx,\bv)) \phantom{\int}\right. \\
&\left.\cdot\int_{0}^{t}(\naX K_n)(\tau,X_n(\tau,t,\bx,\bv))\cdot (\naV X_n)(\tau,t,\bx,\bv)~\intd\tau\right)~\intd\bv.
\end{align*}
Estimating the absolute value by means of Hölder and Minkowski yields
\begin{align*}
\abs{\naX\varphi_{n+1}(t,\bx)} 
\leq&~2 \int_{\RV} \left(\abs{\left(\naZ\mathring
\alpha\right)\circ Z_n(0,t,\bx,\bv)} \bound{\nabla Z}{\infty}(0,t)\phantom\int\right. \\
&+ \left.\abs{\mathring\alpha\circ Z_n(0,t,\bx,\bv)} \int_{0}^{t} \Lpn{\naX K_n(\tau)}{\infty}{\bx}~\bound{\nabla Z}{\infty}(\tau,t)~\intd\tau\right)^2~\intd\bv \\
\leq&~2 \frac{(1+R)^{\kappa}}{\volS{\ddim} R^\ddim}~ \left(\frac{(1+t)R+ \bound{Z-\bar{Z}}{\infty}(0,t)}{1+R}\right)^{\kappa}~ \left(\bound{\nabla Z}{\infty}(0,t) \Apn{\naZ\mathring\alpha}{\kappa,2}{\bz}\phantom{\int}\right. \\
&\left.+ \int_{0}^{t} \Lpn{\naX K_n(\tau)}{\infty}{\bx}~\bound{\nabla Z}{\infty}(\tau,t)~\intd\tau~ \Apn{\mathring\alpha}{\kappa,2}{\bz}\right)^2 \\
\stackrel{R\text{ opt.}}{\leq}&~2 \minR(\kappa,\ddim)~M^2~ \left(1+t+\globalBound{Z-\bar{Z}}{\infty}(0,t)\right)^{\kappa}~ \globalBound{\nabla Z}{\infty}(0,t)^2 \\
&\cdot \left(1+\int_{0}^{t} \Lpn{\naX K_n(\tau)}{\infty}{\bx}~\frac{\globalBound{\nabla Z}{\infty}(\tau,t)}{\globalBound{\nabla Z}{\infty}(0,t)}\intd\tau\right)^{2}.
\end{align*}
Taking the $\ln_+$ of this equation and applying Lemma \ref{lem:newtonian-ineq}, we find
\begin{align*}
&\ln_+\Lpn{\naX\varphi_{n+1}(t)}{\infty}{\bx} \\
\leq& \ln_+\left(2 \minR(\kappa,\ddim)~M^2~ \left(1+t+\globalBound{Z-\bar{Z}}{\infty}(0,t)\right)^{\kappa} \globalBound{\nabla Z}{\infty}(0,t)^2\right)~ \\
&+ 2 \ln_+\left(1+\int_{0}^{t} \Lpn{\naX K_n(\tau)}{\infty}{\bx}~\underbrace{\frac{\globalBound{\nabla Z}{\infty}(\tau,t)}{\globalBound{\nabla Z}{\infty}(0,t)}}_{\leq 1}\intd\tau\right) \\
\leq& \ln_+\left(2 \minR(\kappa,\ddim)~M^2~ \left(1+t+\globalBound{Z-\bar{Z}}{\infty}(0,t)\right)^{\kappa} \globalBound{\nabla Z}{\infty}(0,t)^2\right) \\
&+ 2\int_{0}^{t} \cstNewtonLn \left[\left(1+\globalBound{\varphi}{\infty}(\tau)\right) \left(1+\ln_+\Lpn{\naX\varphi_n(\tau)}{\infty}{\bx}\right)+\globalBound{\varphi}{1}(\tau)\right]\intd\tau.
\end{align*}
An inductive Gronwall argument proves the finiteness of $\sup_n\ln_+\Lpn{\naX\varphi_n(t)}{\infty}{\bx}$, i.e.,
\begin{align*}
&\bound{\nabla\varphi}{\infty}(t) \\
\leq&~ \exp\left(\sup_n \ln_+ \Lpn{\naX\varphi_n(t)}{\infty}{\bx}\right) \\ 
\leq& \exp\left(\left(1 + \ln_+\left(2 \minR(\kappa,\ddim)~M^2~ \left(1+t+\globalBound{Z-\bar{Z}}{\infty}(0,t)\right)^{\kappa} \globalBound{\nabla Z}{\infty}(0,t)^2\right)+2\cstNewtonLn t\globalBound{\varphi}{1}(t)\right)\right. \\
&\left.\cdot\exp\left(\int_{0}^{t} \left(1+\globalBound{\varphi}{\infty}(\tau)\right)\intd\tau\right)-1\right) \equiv \globalBound{\nabla\varphi}{\infty}(t).
\end{align*}
Hence, eqn. \eqref{eqn:D2U-simp-estimate} ensures the finiteness of
\begin{align*}
\bound{\nabla K}{\infty}(t) \leq& \cstNewtonLn\left[\left(1+\globalBound{\varphi}{\infty}(t)\right) \left(1+\ln_+\globalBound{\nabla\varphi}{\infty}\right) + \globalBound{\varphi}{1}(t)\right] \equiv \globalBound{\nabla K}{\infty}(t).
\end{align*} \\

\noindent\textbf{(viii) Uniform supremum and integral bounds on $\naZ\alpha_n$.} For any $(t,\bx,\bv)$ and $n\in\NN_0$ we find
\begin{align*}
&\abs{\naZ\alpha_{n+1}(t,\bx,\bv)} \\
\leq& \abs{(\naZ\mathring\alpha)(Z_n(0,t,\bx,\bv))}\abs{\naZ Z_n(0,t,\bx,\bv)} \\
&+ \abs{\mathring\alpha(Z_n(0,t,\bx,\bv))} \int_{0}^{t} \abs{\naX K_n(\tau,X(\tau,t,\bx,\bv))}\abs{\naZ X_n(\tau,t,\bx,\bv)}\intd\tau \\
\leq& \abs{\left(\naZ\mathring\alpha\right)(Z_n(0,t,\bx,\bv)} \bound{\nabla Z}{\infty}(0,t) + \abs{\mathring\alpha(Z_n(0,t,\bx,\bv))} \int_{0}^{t} \bound{\nabla K}{\infty}(\tau) \bound{\nabla Z}{\infty}(\tau,t)~\intd\tau,
\end{align*}
which yields for the $\Ap{\kappa,2}{\bz}$-norm
\begin{align*}
\bound{\nabla \alpha}{\kappa,2}(t) 
\equiv& \sup_n \Apn{\naZ\alpha_n(t)}{\kappa,2}{\bz} \\
\leq&  \left(1+t+\bound{Z-\bar{Z}}{\infty}(0,t)\right)^{\frac{\kappa}{2}} \left(\Apn{\naZ\mathring\alpha}{\kappa,2}{\bz} \bound{\nabla Z}{\infty}(0,t) + \Apn{\mathring\alpha}{\kappa,2}{\bz} \int_{0}^{t} \bound{\nabla K}{\infty}(\tau) \bound{\nabla Z}{\infty}(\tau,t)~\intd\tau\right) \\
\leq& \left(1+t+\globalBound{Z-\bar{Z}}{\infty}(0,t)\right)^{\frac \kappa 2} M\left(\globalBound{\nabla Z}{\infty}(0,t) + \int_{0}^{t} \globalBound{\nabla K}{\infty}(\tau) \globalBound{\nabla Z}{\infty}(\tau,t)\intd\tau\right) \\
\equiv& \globalBound{\nabla\alpha}{\kappa,2}(t)
\end{align*}
proving the finiteness on the interval $I$. This proves that $\bound{\alpha}{1,\kappa,2}(t) \equiv \sup_n \Bpn{\mathring\alpha}{1,\kappa,2}{\bz}<\infty$ in $I$ and the curve of each convergent remains in $\Bp{1,\kappa,2}{\bz}$.
\\

\noindent\textbf{(ix) Difference estimates.} We now want to show that $\alpha_n$ converges to some limit $\alpha$ in $\Lp{\infty}{t}\Ap{\kappa,2}{\bz}$ norm. At first, we estimate the difference $\alpha_{n+1}(t)-\alpha_n(t)$ of convergents at $\bz\in\RZ$
\begin{align*}
&\abs{\alpha_{n+2}(t,\bz)-\alpha_{n+1}(t,\bz)} \\
\leq& \abs{\mathring\alpha(Z_{n+1}(0,t,\bz))-\mathring\alpha(Z_n(0,t,\bz))} \abs{\exp\left(\int_{0}^{t} K_{n+1}(\tau,X_{n+1}(\tau,t,\bz))~\intd\tau\right)} \\
&+ \abs{\mathring\alpha(Z_n(0,t,\bz))} \abs{\exp\left(\int_{0}^{t} K_{n+1}(\tau,X_{n+1}(\tau,t,\bz))~\intd\tau\right)-\exp\left(\int_{0}^{t} K_n(\tau,X_n(\tau,t,\bz))~\intd\tau\right)} \\
\leq& \int_{0}^{1} \abs{\left(\naZ\mathring\alpha\right)((1-s) Z_n(0,t,\bz) + s Z_{n+1}(0,t,\bz))}~\intd s \abs{Z_{n+1}(0,t,\bz)-Z_n(0,t,\bz)} \\
&+ \abs{\mathring\alpha(Z_n(0,t,\bz))} \abs{\int_{0}^{t} \left(K_{n+1}(\tau,X_{n+1}(\tau,t,\bz))-K_n(\tau,X_{n+1}(\tau,t,\bz))\right)~\intd\tau} \\
&+ \abs{\mathring\alpha(Z_n(0,t,\bz))} \abs{\int_{0}^{t} \left(K_{n}(\tau,X_{n+1}(\tau,t,\bz))-K_n(\tau,X_{n}(\tau,t,\bz))\right)~\intd\tau} \\
\leq& \int_{0}^{1} \abs{\left(\naZ\mathring\alpha\right)((1-s) Z_n(0,t,\bz) + s Z_{n+1}(0,t,\bz))}~\intd s \Lpn{Z_{n+1}(0,t)-Z_n(0,t)}{\infty}{\bz}\\
&+ \abs{\mathring\alpha(Z_n(0,t,\bz))} \int_{0}^{t} \Lpn{K_{n+1}(\tau)-K_n(\tau)}{\infty}{\bx}\intd\tau \\
&+ \abs{\mathring\alpha(Z_n(0,t,\bz))} \int_{0}^{t} \Lpn{\naX K_n(\tau)}{\infty}{\bx} \Lpn{X_{n+1}(\tau,t)-X_n(\tau,t)}{\infty}{\bz}\intd\tau.
\end{align*}
Now we can treat all the terms differently. Primarily, we want to remark that by Lemma \ref{lem:symplectomorphism-difference}
\begin{align*}
&\sup_n \Lpn{(1-s)(Z_{n}(0,t)-\bar{Z}(0,t)) + s(Z_{n+1}(0,t)-\bar{Z}(0,t))}{\infty}{\bz} \\
\leq&~ (1-s) \sup_n \Lpn{Z_n(0,t)-\bar{Z}(0,t)}{\infty}{\bz} + s \sup_n \Lpn{Z_{n+1}(0,t)-\bar{Z}(0,t)}{\infty}{\bz} \\ 
\leq& \sup_{n} \Lpn{Z_n(0,t)-\bar{Z}(0,t)}{\infty}{\bz} = \bound{Z-\bar{Z}}{\infty}(0,t) \leq \globalBound{Z-\bar{Z}}{\infty}(0,t),
\end{align*}
yielding for any $R>0$
\begin{align*}
& \int_{\RZ} \sup_{\abs{\bar{\bz}}\leq R} \left(\int_{0}^{1} \abs{\left(\naZ\mathring\alpha\right)(s Z_n(0,t,\bz+\bar{\bz}) + (1-s) Z_{n+1}(0,t,\bz+\bar{\bz}))}~\intd s\right)^2 \intd\bz \\
\leq& \int_{\RZ} \sup_{\abs{\bar{\bz}}\leq R} \left(\sup_{\abs{\bar{\bar{\bz}}}\leq \globalBound{Z-\bar{Z}}{\infty}(0,t)} \abs{\left(\naZ\mathring\alpha\right)(\bar{Z}(0,t,\bz+\bar{\bz}) + \bar{\bar{\bz}})}\right)^2 \intd\bz \\
\leq& \int_{\RZ} \sup_{\abs{\bar{\bz}}\leq (1+t)R + \globalBound{Z-\bar{Z}}{\infty}(0,t)} \abs{\left(\naZ\mathring\alpha\right)(\bar{Z}(0,t,\bz)+\bar{\bz})}^2 \intd\bz \\
=& \int_{\RZ} \sup_{\abs{\bar{\bz}}\leq (1+t)R + \globalBound{Z-\bar{Z}}{\infty}(0,t)} \abs{\left(\naZ\mathring\alpha\right)(\bz+\bar{\bz})}^2 \intd\bz.
\end{align*}
Dividing by $(1+R)^{\kappa}$ and applying the supremum and the square root, we find
\begin{align*}
& \sup_{R\geq 0} \left(1+R\right)^{-\frac{\kappa}{2}} \\ &\cdot \left(\int_{\RZ} \sup_{\abs{\bar{\bz}}\leq R} \left(\int_{0}^{1} \abs{\left(\naZ\mathring\alpha\right)(s Z_n(0,t,\bz+\bar{\bz}) + (1-s) Z_{n+1}(0,t,\bz+\bar{\bz}))}~\intd s\right)^2 \intd\bz\right)^{\frac 12} \\
\leq& \left(1+t+\globalBound{Z-\bar{Z}}{\infty}(0,t)\right)^{\frac{\kappa}{2}} \Apn{\naZ\mathring\alpha}{\kappa,2}{\bz}.
\end{align*}
For the second and third term, we find similarly
\begin{align*}
\sup_{R\geq 0} \left(1+R\right)^{-\frac \kappa 2} \left(\int_{\RZ} \sup_{\abs{\bar{\bz}}\leq R} \abs{\mathring\alpha(Z_n(0,t,\bz+\bar{\bz}))}^2 \intd\bz\right)^{\frac 12} \leq& \left(1+t+\globalBound{Z-\bar{Z}}{\infty}(0,t)\right)^{\frac \kappa 2} \Apn{\mathring\alpha}{\kappa,2}{\bz}.
\end{align*}
Upon combining these estimates with the Lipschitz estimates of Lemma \ref{lem:local-lipschitz}, one derives
\begin{align*}
&\left(1+t+\globalBound{Z-\bar{Z}}{\infty}(0,t)\right)^{-\frac \kappa 2} \Apn{\alpha_{n+2}(t)-\alpha_{n+1}(t)}{\kappa,2}{\bz} \\
\leq& \Apn{\naZ\mathring\alpha}{\kappa,2}{\bz} \Lpn{Z_{n+1}(0,t)-Z_n(0,t)}{\infty}{\bz} + \Apn{\mathring\alpha}{\kappa,2}{\bz} \int_{0}^{t} \Lpn{K_{n+1}(\tau)-K_n(\tau)}{\infty}{\bx} \intd\tau \\
&+ \Apn{\mathring\alpha}{\kappa,2}{\bz} \int_{0}^{t} \Lpn{\naX K_n(\tau)}{\infty}{\bx} \Lpn{X_{n+1}(\tau,t)-X_n(\tau,t)}{\infty}{\bz} \intd\tau \\
\leq& \Apn{\naZ\mathring\alpha}{\kappa,2}{\bz} \int_{0}^{t} \Lpn{F_{n+1}(\tau)-F_n(\tau)}{\infty}{\bx} \exp\left(\int_{0}^{\tau} \left(1+\Lpn{\naX F_n(\tilde{\tau})}{\infty}{\bx}\right)\intd\tilde{\tau}\right) \intd\tau \\
&+ \Apn{\mathring\alpha}{\kappa,2}{\bz} \int_{0}^{t} \Lpn{K_{n+1}(\tau)-K_n(\tau)}{\infty}{\bx} \intd\tau \\
&+ \Apn{\mathring\alpha}{\kappa,2}{\bz} \int_{0}^{t} \Lpn{\naX K_n(\tau)}{\infty}{\bx} \int_{\tau}^{t} \Lpn{F_{n+1}(\tilde{\tau})-F_n(\tilde{\tau})}{\infty}{\bx} \\
&\quad\cdot \exp\left(\int_{\tau}^{\tilde{\tau}} \left(1+\Lpn{\naX F_n(\tilde{\tilde{\tau}})}{\infty}{\bx}\right)\intd\tilde{\tilde{\tau}}\right) \intd\tilde{\tau} \intd\tau \\
\leq& \Apn{\naZ\mathring\alpha}{\kappa,2}{\bz} \int_{0}^{t} \Lpn{F_{n+1}(\tau)-F_n(\tau)}{\infty}{\bx} \exp\left(\int_{0}^{\tau} \left(1+\globalBound{\nabla F}{\infty}(\tilde{\tau})\right)\intd\tilde{\tau}\right) \intd\tau \\
&+ \Apn{\mathring\alpha}{\kappa,2}{\bz} \int_{0}^{t} \Lpn{K_{n+1}(\tau)-K_n(\tau)}{\infty}{\bx} \intd\tau \\
&+ \Apn{\mathring\alpha}{\kappa,2}{\bz} \int_{0}^{t} \Lpn{F_{n+1}(\tau)-F_n(\tau)}{\infty}{\bx} \\
&\quad\cdot\left(\int_{0}^{\tau} \globalBound{\nabla K}{\infty}(\tilde{\tau}) \exp\left(\int_{\tilde{\tau}}^{\tau} \left(1+\globalBound{\nabla F}{\infty}(\tilde{\tilde{\tau}})\right)\intd\tilde{\tilde{\tau}}\right) \intd\tilde{\tau}\right) \intd\tau \\
\leq&~ M \int_{0}^{t} \Lpn{K_{n+1}(\tau)-K_n(\tau)}{\infty}{\bx} \intd\tau \\
&+ M\int_{0}^{t} \Lpn{F_{n+1}(\tau)-F_n(\tau)}{\infty}{\bx} \left(\exp\left(\int_{0}^{\tau} \left(1+\globalBound{\nabla F}{\infty}(\tilde{\tau})\right)\intd\tilde{\tau}\right)\right. \\
&\left.+  \int_{0}^{\tau} \globalBound{\nabla K}{\infty}(\tilde{\tau}) \exp\left(\int_{\tilde{\tau}}^{\tau} \left(1+\globalBound{\nabla F}{\infty}(\tilde{\tilde{\tau}})\right)\intd\tilde{\tilde{\tau}}\right) \intd\tilde{\tau}\right) \intd\tau\\
\leq&~ \cstNewtonLp{1} \minR\left(\kappa,\ddim\right)^{1-\frac 1\ddim} M \int_{0}^{t} \left(\Apn{\naV\alpha_n(\tau)}{\kappa,2}{\bz} + \Apn{\naV\alpha_{n+1}(\tau)}{\kappa,2}{\bz}\right) \Apn{\alpha_{n+1}(\tau)-\alpha_n(\tau)}{\kappa,2}{\bz} \intd\tau \\
&+ \cstNewtonLp{1} \minR(\kappa,\ddim)^{1-\frac 1\ddim} M \int_{0}^{t} \left(\Apn{\alpha_n(\tau)}{\kappa,2}{\bz} + \Apn{\alpha_{n+1}(\tau)}{\kappa,2}{\bz}\right) \Apn{\alpha_{n+1}(\tau)-\alpha_n(\tau)}{\kappa,2}{\bz} \\
&\quad\cdot \left(\exp\left(\int_{0}^{\tau} \left(1+\globalBound{\nabla F}{\infty}(\tilde{\tau})\right)\intd\tilde{\tau}\right) \right. \\
&\left.\quad\quad + \int_{0}^{\tau} \globalBound{\nabla K}{\infty}(\tilde{\tau}) \exp\left(\int_{\tilde{\tau}}^{\tau} \left(1+\globalBound{\nabla F}{\infty}(\tilde{\tilde{\tau}})\right)\intd\tilde{\tilde{\tau}}\right) \intd\tilde{\tau}\right) \intd\tau \\
\leq&~ \globalBound{1}{}(t) \int_{0}^{t} \globalBound{2}{}(\tau) \Apn{\alpha_{n+1}(\tau)-\alpha_n(\tau)}{\kappa,2}{\bz} \intd\tau,
\end{align*}
where
\begin{align*}
\globalBound{1}{}(t) \equiv& \left(1+t+\globalBound{Z-\bar{Z}}{\infty}(0,t)\right)^{\frac \kappa 2}, \\
\globalBound{2}{}(\tau) \equiv&~ 2\cstNewtonLp{1} \minR(\kappa,\ddim)^{1-\frac 1\ddim} M \left[\globalBound{\nabla\alpha}{\kappa,2}(\tau) + \globalBound{\alpha}{\kappa,2}(\tau) \left(\exp\left(\int_{0}^{\tau} \left(1+\globalBound{\nabla F}{\infty}(\tilde{\tau})\right)\intd\tilde{\tau}\right)\right.\right. \\
&\left.\left.+  \int_{0}^{\tau} \globalBound{\nabla K}{\infty}(\tilde{\tau}) \exp\left(\int_{\tilde{\tau}}^{\tau} \left(1+\globalBound{\nabla F}{\infty}(\tilde{\tilde{\tau}})\right)\intd\tilde{\tilde{\tau}}\right) \intd\tilde{\tau}\right)\right].
\end{align*}
Upon estimating $\globalBound{2}{}(\tau)\leq \globalBound{2}{}(t)$ for $\tau\leq t$, we even deduce for $\globalBound{1+2}{}(t) \equiv \globalBound{1}{}(t)~ \globalBound{2}{}(t)$
\begin{align*}
\Apn{\alpha_{n+2}(t)-\alpha_{n+1}(t)}{\kappa,2}{\bz} \leq& \globalBound{1+2}{}(t) \int_{0}^{t} \Apn{\alpha_{n+1}(\tau)-\alpha_n(\tau)}{\kappa,2}{\bz} \intd\tau.
\end{align*}
By induction, one immediately verifies that
\begin{align*}
\sup_{\tau\in[0,t]} \Apn{\alpha_{n+1}(\tau)-\alpha_n(\tau)}{\kappa,2}{\bz} 
\leq&~ 2 \sup_{\tau\in[0,t]}\bound{\alpha}{\kappa,2}(\tau) \frac{\globalBound{1+2}{}(t)^n t^n}{n!} \leq 2\globalBound{\alpha}{\kappa,2}(t) \frac{\globalBound{1+2}{}(t)^n t^n}{n!}.
\end{align*}
As the right-hand side is summable in $n$ for any $t<T(M)$, the sequence $\alpha_n$ converges uniformly on $[0,t]\times\RX\times\RV$ to some limit $\alpha\in\Cm{0}{I\times\RX\times\RV\to\CN}$ and $\alpha \equiv \lim_{n\to\infty}\alpha_n \in\Lp{\infty}{t}\Ap{\kappa,2}{\bz}$. \\

\noindent\textbf{(x) Regularity.} Using Lemmata \ref{lem:local-lipschitz} and \ref{lem:symplectomorphism-difference}, we can estimate the Cauchy property of all the sequences $Z_n, \rho_n, F_n, \varphi_n, K_n$ uniformly on the set $[0,t]\times\RZ$ for any $t\in I$. In addition, $\rho_n$ and $\varphi_n$ are Cauchy sequences in the $\Lp{1}{\bx}$ norm. As the space of continuous functions with the $\Lp{\infty}{}$ norm is complete, we find the continuous limits $Z, \rho, F, \varphi, K$. We need to show that these limits are continuously differentiable. This is achieved by proving the uniform Cauchy property of the first derivatives. 

We start by examining $F_n$. Let $0\leq\tau\leq t<T$ and $\epsilon>0$ be arbitrary. For $R=\epsilon/\left(4\cstNewtonLn\globalBound{\nabla\rho}{\infty}(t)\right)$ in the estimate \eqref{eqn:D2U-estimate} applied on $F_n(\tau)-F_m(\tau)$, one finds, regarding the monotonicity of $\globalBound{\nabla\rho}{\infty}(\cdot)$,
\begin{align*}
&\Lpn{\naX F_n(\tau)-\naX F_m(\tau)}{\infty}{\bx} \\
\leq&~ \cstNewtonLn\left[\left(1+\ln\frac rR\right)\Lpn{\rho_n(\tau)-\rho_m(\tau)}{\infty}{\bx}+\frac{1}{r^\ddim} \Lpn{\rho_n(\tau)-\rho_m(\tau)}{1}{\bx} \right.\\
&\quad\left.\phantom{\frac{}{}}+ R\Lpn{\naX\rho_n(\tau)-\naX\rho_m(\tau)}{\infty}{\bx}\right] \\
\stackrel{\text{Lem.\ref{lem:local-lipschitz}}}{\leq}&~ \left[\cstNewtonLn\left(1+\ln\frac rR\right) \minR(\kappa,\ddim) + \frac{\cstNewtonLn}{r^\ddim}\right] 2\globalBound{\alpha}{\kappa,2}(\tau) \Apn{\alpha_n(\tau)-\alpha_m(\tau)}{\kappa,2}{\bz} + \frac\epsilon 2.
\end{align*}
where the first term can be made $\leq\frac{\epsilon}{2}$ choosing $n,m$ large enough by the Cauchy property of $\alpha_n\in\Ap{\kappa,2}{\bz}$ and the estimates from (ix). Hence, $\naX F_n$ is uniformly Cauchy on $[0,t]\times\RX$ and converges to a continuous limit, that is known to be $\naX F$ and $F$ is continuously differentiable. An identical argument proves the existence and continuity of $\naX K$. 

Now taking the limit $n\to\infty$ of
\begin{align*}
Z_n(s,t,\bz) = \bz + \int_{t}^{s} \left(V_n(\tau,t,\bz),F_n(\tau,X_n(\tau,t,\bz))\right)~\intd\tau
\end{align*}
proves $Z$ to solve an ODE system with right-hand side $(t,\bx,\bv)\mapsto (\bv,F(t,\bx))$. The continuity of $\naX F$ proves $Z\in\Cm{1}{I\times I\times\RX\times\RV}$ by standard theory \cite{arnold}.

Combining these results, one finds
\begin{align}
\label{eqn:alpha-construction}
\alpha(t,\bz) = \mathring\alpha(Z(0,t,\bz))~\exp\left(\int_{0}^{t} K(\tau,X(\tau,t,\bz))~\intd\tau\right)
\end{align}
to be continuously differentiable as a composition of such functions. \\

\noindent\textbf{(xi) Solution.} This is a consequence of Lemma \ref{lem:transport-formula} and eqn. \eqref{eqn:alpha-construction}. \qedhere
\end{proof}

\begin{cor}
\label{cor:bounds}
Let $\alpha\in \Cm{1}{I\times\RZ; \CN}$ be a local solution of the Hamiltonian Vlasov-Poisson equation with initial datum $\alpha(0)=\mathring\alpha\in\ball{\bn}{M}\subseteq\Bp{1,\kappa,2}{\bz}$. Then all the bounds $\globalBound{\cdot}{\cdot}$ derived in the proof of Lemma \ref{lem:local-solution} hold also for the characteristic tuple of $\alpha$.
\end{cor}
\begin{proof}
As $\alpha$ is a solution, the transport formula \eqref{eqn:transport-formula} holds. In turn this implies that $\alpha$ and the characteristic tuple is a fixed point of the iterative scheme in Lemma \ref{lem:local-solution}. Therefore, all the inductive Gronwall arguments used to prove Lemma \ref{lem:local-solution} can be replaced by non-inductive Gronwall estimates that deliver the same bounds. \qedhere
\end{proof}

\begin{cor}
\label{cor:solution-distance}
Let $\alpha_1, \alpha_2 \in\Cm{1}{I\times\RX\times\RV}$ be two local solutions of the Hamiltonian Vlasov-Poisson equation with initial data $\mathring\alpha_1,\mathring\alpha_2\in\ball{\bn}{M}\subseteq\Bp{1,\kappa,2}{\bz}$. Then
\begin{equation}
\label{eqn:solution-stability}
\Apn{\alpha_1(t)-\alpha_2(t)}{\kappa,2}{\bz} \leq \Apn{\mathring\alpha_1-\mathring\alpha_2}{\kappa,2}{\bz} \left(1+t+\globalBound{Z-\bar{Z}}{\infty}(0,t)\right)^{\frac\kappa 2} \exp\left(\int_{0}^{t} \globalBound{1+2}{}(\tau) \intd\tau\right).
\end{equation}
\end{cor}
\begin{proof}
By Corollary \ref{cor:bounds}, all the bounds for the solutions $\alpha_1,\alpha_2$ are uniformly valid. We remark that by the transport formula \eqref{eqn:transport-formula} the solutions are fixed points of the iterative scheme. Therefore, we can estimate
\begin{align*}
\abs{\alpha_1(t,\bz)-\alpha_2(t,\bz)} \leq& \abs{\mathring\alpha_1(Z_1(0,t,\bz))-\mathring\alpha_2(Z_1(0,t,\bz))} + \abs{\mathring\alpha_2(Z_1(0,t,\bz))-\mathring\alpha_2(Z_2(0,t,\bz))} \\
&+ \abs{\mathring\alpha_2(Z_2(0,t,\bz))} \int_{0}^{t} \abs{K_1(\tau,X_1(\tau,t,\bz))-K_2(\tau,X_1(\tau,t,\bz))} \intd\tau \\
&+ \abs{\mathring\alpha_2(Z_2(0,t,\bz))} \int_{0}^{t} \abs{K_2(\tau,X_1(\tau,t,\bz))-K_2(\tau,X_2(\tau,t,\bz))} \intd\tau.
\end{align*}
All terms except for the first one can be treated exactly as in the proof of Lemma \ref{lem:local-solution}. For the first term, we find in the $\Ap{\kappa,2}{\bz}$-norm
\begin{align*}
&\sup_{R\geq 0} (1+R)^{-\frac{\kappa}{2}} \left(\int_\RZ \sup_{\abs{\bar{\bz}}\leq R} \abs{\left(\mathring\alpha_1-\mathring\alpha_2\right)(Z_1(0,t,\bz+\bar{\bz}))}^2 \intd\bz\right)^{\frac{1}{2}} \\
\leq& \sup_{R\geq 0} (1+R)^{-\frac{\kappa}{2}} \left(\int_\RZ \sup_{\abs{\bar{\bz}}\leq (1+t)R + \globalBound{Z-\bar{Z}}{\infty}(0,t)} \abs{\left(\mathring\alpha_1-\mathring\alpha_2\right)(\bar{Z}(0,t,\bz)+\bar{\bz})}^2 \intd\bz\right)^{\frac{1}{2}} \\
\leq& \left(1+t+\globalBound{Z-\bar{Z}}{\infty}(0,t)\right)^{\frac{\kappa}{2}} \Apn{\mathring\alpha_1-\mathring\alpha_2}{\kappa,2}{\bz} = \globalBound{1}{}(t) \Apn{\mathring\alpha_1-\mathring\alpha_2}{\kappa,2}{\bz}.
\end{align*}
Combining this estimate with the estimates of the other terms yields
\begin{align*}
\Apn{\alpha_1(t)-\alpha_2(t)}{\kappa,2}{\bz} 
\leq&~ \globalBound{1}{}(t) \Apn{\mathring\alpha_1-\mathring\alpha_2}{\kappa,2}{\bz} \\
&\quad+ \globalBound{1}{}(t) \int_{0}^{t} \globalBound{2}{}(\tau) \Apn{\alpha_1(\tau)-\alpha_2(\tau)}{\kappa,2}{\bz} \intd\tau.
\end{align*}
A short Gronwall application delivers the desired inequality. \qedhere
\end{proof}

\begin{prop}
\label{prop:initial-datum-approximation}
Let $\left\{\mathring\alpha_n\right\}\subseteq\ball{\bn}{M}\subseteq\Bp{1,\kappa,2}{\bz}$, $\Bpn{\mathring{\alpha}_n-\mathring{\alpha}}{1,\kappa,2}{\bz}\to 0$, be a convergent sequence of compactly supported initial data with solutions $\alpha_n$ on $I(M)=[0,T(M))$. Then there is a solution $\alpha$ of the Hamiltonian Vlasov-Poisson equation for the initial datum $\mathring\alpha$ and for every $t<T(M)$
\begin{equation*}
\sup_{\tau\in[0,t]} \Apn{\alpha(\tau)-\alpha_n(\tau)}{\kappa,2}{\bz} \stackrel{n\to\infty}{\to} 0.
\end{equation*}
\end{prop}
\begin{proof}
\textbf{(i) Limit.} At first, we remark that $\{\mathring\alpha_n\}$ is Cauchy in $\Bp{1,\kappa,2}{\bz}$ and hence is $\{\alpha_n\}$ in $\Lp{\infty}{t}\Ap{\kappa,2}{\bz}$ by Corollary \ref{cor:solution-distance} if one restricts $t$ on any compact subinterval $J\subseteq I(M)$. As all the $\alpha_n$ are continuous, there is a limit $\alpha\in\Cm{0}{I(M)\times\RZ;\CN}$, s.t.
\begin{align*} \sup_{t\in J} \Apn{\alpha(t)-\alpha_n(t)}{\kappa,2}{\bz}\stackrel{n\to\infty}{\longrightarrow} 0.
\end{align*} \\

\noindent\textbf{(ii) Tuple convergence.} We now want to prove that $\{\rho_n\}, \{F_n\}$ converge to density $\rho$ and force $F$ induced by the limit $\alpha$. In fact, this is an immediate consequence of Lemma \ref{lem:local-lipschitz}-(i)-(iii). We remark that even $\naX F$ exists and is continuous, because $\{\naX F_n\}$ is a Cauchy sequence in $\Lp{\infty}{(t,\bx)}$. We argue similarly to (x) in the proof of Lemma \ref{lem:local-solution}. By Corollary \ref{cor:bounds}, for all $n$ $\sup_{\tau\leq t} \Lpn{\naX\rho_n(\tau)}{\infty}{\bx} \leq \globalBound{\nabla \rho}{\infty}(t)$. By \eqref{eqn:D2U-estimate} in Lemma \ref{lem:newtonian-ineq}, we have
\begin{align*}
\Lpn{\naX F_n(t)-\naX F_m(t)}{\infty}{\bx}
\leq&~ \cstNewtonLn \left[\left(1+\ln\frac r R\right)\Lpn{\rho_n(t)-\rho_m(t)}{\infty}{\bx} \right.\\
&\quad\left.\phantom{\frac{}{}}+ \frac 1{r^\ddim} \Lpn{\rho_n-\rho_m}{1}{\bx} + R\Lpn{\naX\rho_n(t)-\naX\rho_m(t)}{\infty}{\bx}\right].
\end{align*}
For any given $\epsilon>0$, we choose $R=\epsilon/(4\cstNewtonLn\globalBound{\nabla\rho}{\infty}(t))$ and some $r\geq R$. The first two terms are controlled by Lemma \ref{lem:local-lipschitz}-(i)-(ii) and the Cauchy property of $\alpha_n$. 

For $\varphi$, we only need to remark that for all $n$ $\sup_{\tau\leq t} \Apn{\naZ\alpha_n(\tau)}{\kappa,2}{\bz}\leq \globalBound{\nabla\alpha}{\kappa,2}(t)$. By Lemma \ref{lem:local-lipschitz}-(iv)-(vi), $\varphi_n$ and $K_n$ will be Cauchy in their respective spaces. Denote the limits by $\varphi_\infty$ and $K_\infty$. After noticing $\sup_n \sup_{\tau\in[0,t]} \Lpn{\naX\varphi_n(\tau)}{\infty}{\bx} \leq \globalBound{\nabla\varphi}{\infty}(t)$, the argument of $\naX F$ can be copied to prove the convergence $\Lpn{\naX K_n(t)-\naX K_\infty(t)}{\infty}{\bx}\to 0$. 

Now let $Z_n$ denote the solution map of the characteristic system. We find by Lemma \ref{lem:symplectomorphism-difference}
\begin{align*}
\Lpn{Z_n(s,t)-Z_m(s,t)}{\infty}{\bz} \leq \int_{s}^{t} \Lpn{F_n(\tau)-F_m(\tau)}{\infty}{\bx} \exp\left(\int_{s}^{\tau} \left(1+\globalBound{\nabla F}{\infty}(\tilde{\tau})\right) \intd\tilde{\tau}\right)\intd\tau,
\end{align*}
proving the uniform Cauchy property of $Z_n$ and thus its convergence. This allows to compute
\begin{align*}
Z(s,t,\bz) \equiv& \lim_n Z_n(s,t,\bz) = \lim_n (X_n,V_n)(s,t,\bz) \\
=&~\bz + \lim_n \int_{t}^{s} (V_n(\tau,t,\bz),F_n(\tau,X_n(\tau,t,\bz))) \intd\tau \\
=&~\bz + \int_{t}^{s} (V(\tau,t,\bz),F(\tau,X(\tau,t,\bz))) \intd\tau.
\end{align*}
Hence, $Z$ is the solution map of a characteristic system with $\Cp{0}{t}\Cp{1}{\bz}$ right-hand side. By standard theory \cite{arnold}, $Z\in\Cp{1}{}$.
\\

\noindent\textbf{(iii) Solution.} As all the $\alpha_n$ are solutions, they fulfill the transport formula
\begin{equation*}
\alpha_n(t,\bz) = \mathring\alpha_n(Z_n(0,t,\bz)) \exp\left(\int_{0}^{t} K_n(\tau,X_n(\tau,t,\bz)) \intd\tau\right).
\end{equation*}
In the previous part we have shown that all the sequences converge for the limit $n\to\infty$ sufficiently nice to derive
\begin{equation*}
\alpha(t,\bz) = \mathring\alpha(Z(0,t,\bz)) \exp\left(\int_{0}^{t} K_\infty(\tau,X(\tau,t,\bz)) \intd\tau\right)
\end{equation*}
and indeed, all the right-hand side is $\Cp{1}{(t,\bz)}$. $\alpha$ will be a local solution for the initial datum $\mathring\alpha$ by Lemma \ref{lem:transport-formula} if $(\rho,F,\varphi_\infty,K_\infty)$ is indeed its characteristic tuple. At first, we verify that $\varphi$ is well-defined. In fact,
\begin{align*}
\Apn{\naZ\alpha(t)}{\kappa,2}{\bz} 
\leq& \left(1+t+\Lpn{Z(0,t)-\bar{Z}(0,t)}{\infty}{\bz}\right)^{\frac{\kappa}{2}} \\
&\cdot \left(\Apn{\naZ\mathring\alpha}{\kappa,2}{\bz} \Lpn{\naZ Z(0,t)}{\infty}{\bz} + \Apn{\mathring\alpha}{\kappa,2}{\bz} \int_{0}^{t} \Lpn{\naX K_\infty(\tau)}{\infty}{\bx} \Lpn{\naZ Z(\tau,t)}{\infty}{\bz}\intd\tau\right) \\
\leq&~2 \left(1+t+\Lpn{Z(0,t)-\bar{Z}(0,t)}{\infty}{\bz}\right)^{\frac{\kappa}{2}} \\
&\cdot \left(\Apn{\naZ\mathring\alpha}{\kappa,2}{\bz} \exp\left(\int_{0}^{t} \left(1+\Lpn{\naX F(\tau)}{\infty}{\bx}\right) \intd\tau\right)\right. \\
&\left.+ \Apn{\mathring\alpha}{\kappa,2}{\bz} \int_{0}^{t} \Lpn{\naX K_\infty(\tau)}{\infty}{\bx} \exp\left(\int_{\tau}^{t} \left(1+\Lpn{\naX F(\tilde{\tau})}{\infty}{\bx}\right) \intd\tilde{\tau}\right) \intd\tau\right) \\
\leq&~2 \left(1+t+\globalBound{Z-\bar{Z}}{\infty}(0,t)\right)^{\frac{\kappa}{2}} \left(\Apn{\naZ\mathring\alpha}{\kappa,2}{\bz} \exp\left(\int_{0}^{t} \left(1+\globalBound{\nabla F}{\infty}(\tau)\right) \intd\tau\right)\right. \\
&\left.+ \Apn{\mathring\alpha}{\kappa,2}{\bz} \int_{0}^{t} \globalBound{\nabla K}{\infty}(\tau) \exp\left(\int_{\tau}^{t} \left(1+\globalBound{\nabla F}{\infty}(\tilde{\tau})\right) \intd\tilde{\tau}\right) \intd\tau\right),
\end{align*}
proving that $\alpha(t)$ has a well-defined characteristic tuple. Now for the phase force $K$, we indeed find
\begin{align*}
\Lpn{K(t)-K_\infty(t)}{\infty}{\bx} \leq& \Lpn{K(t)-K_n(t)}{\infty}{\bx} + \Lpn{K_\infty(t)-K_n(t)}{\infty}{\bx} \\
\leq&~ \cstNewtonLp{1}^{1-\frac 1\ddim} \minR(\kappa,\ddim)^{1-\frac 1\ddim} \left(\Apn{\naV\alpha(t)}{\kappa,2}{\bz}+\Apn{\naV\alpha_n(t)}{\kappa,2}{\bz}\right) \Apn{\alpha(t)-\alpha_n(t)}{\kappa,2}{\bz} \\
&+ \Lpn{K_\infty(t)-K_n(t)}{\infty}{\bx} \stackrel{n\to\infty}{\to} 0.
\end{align*}
Therefore, $\alpha$ satisfies the transport formula and is a solution by Lemma \ref{lem:transport-formula} with initial datum $\alpha(0)=\mathring\alpha$. \qedhere
\end{proof}

\begin{prop}[Well-posedness]
\label{prop:time-evolution-operator}
For every $M>0$ and $t<T(M)$ there is a Lipschitz continuous time evolution operator
\begin{equation}
\label{eqn:time-evolution-operator}
\EvoOp{M}{t}: \begin{array}{ccc}
\ball{\bn}{M} \subseteq \Bp{1,\kappa,2}{\bz} & \to & \Ap{\kappa,2}{\bz} \\
\mathring\alpha & \mapsto & \alpha(t),
\end{array}
\end{equation}
that maps the initial datum to the state at time $t$. The function $\alpha(t,\bz) \equiv \left(\EvoOp{M}{t}\mathring\alpha\right)(\bz)$ then is the unique local solution of the Hamiltonian Vlasov-Poisson equation with initial datum $\mathring\alpha$. Indeed, we have $\EvoOp{M}{t}(\ball{\bn}{M})\subseteq \Bp{1,\kappa,2}{\bz}\subseteq \Ap{\kappa,2}{\bz}$.
\end{prop}
\begin{proof}
Let $M>0$ be arbitrary. At first, restrict the domain of definition to the set of compactly supported initial data. By Lemma \ref{lem:local-solution}, there is the uniform time of existence $T(M)$, s.t. there is a solution for any of these initial data. By Corollary \ref{cor:solution-distance}, the time evolution operator is globally Lipschitz on the compactly supported data in $\ball{\bn}{M}\subseteq \Bp{1,\kappa,2}{\bz}$. By Lemma \ref{lem:compact-dense}, these initial data are dense. Because $\Bp{1,\kappa,2}{\bz}$ is complete by Theorem \ref{thm:ap-banach}, there is a unique Lipschitz continuous extension of $\EvoOp{M}{t}$ on $\ball{\bn}{M}\subseteq \Bp{1,\kappa,2}{\bz}$. Proposition \ref{prop:initial-datum-approximation} finally proves that the functions obtained by this technical extension truly are solutions of the Hamiltonian Vlasov-Poisson equation.

The uniqueness of these solutions is immediately provided by Corollary \ref{cor:solution-distance}, as any two solutions with same initial datum will coincide.

The final statement is obvious once we recognize that the finiteness of $\Bpn{\alpha(t)}{1,\kappa,2}{\bz}$ was estimated up to $T(M)$. \qedhere
\end{proof}

\subsection{Global existence}

Due to the close relation with the classical Vlasov-Poisson equation, the Hamiltonian Vlasov-Poisson equation admits almost the same continuation and globality criteria due to Pfaffelmoser-Schaeffer \cite{pfaffelmoser,schaeffer} and Lions-Perthame \cite{lionsperthame}. For completeness, we include the leading steps before adapting their central results.

\begin{prop}[Maximality criterion]
Let $\alpha\in\Cm{1}{[0,T)\times\RZ;\CN}$ be a local solution of the Hamiltonian Vlasov-Poisson equation with initial datum $\mathring{\alpha}\in\Bp{1,\kappa,2}{\bz}$. It is maximal if either $T=\infty$ or $\Bpn{\alpha(t)}{1,\kappa,2}{\bz} \to \infty$ for $t\uparrow T$. In any other case, there exists an extension. In particular, we find a maximal solution for every initial datum $\mathring{\alpha}$.
\end{prop}
\begin{proof}
Maximality is obvious in both cases. Hence assume, none of the cases holds. By definition, this yields
\begin{equation*}
\exists M\geq 0: \forall \bar{T} \in[0,T)~\exists t\in(\bar{T},T): \Bpn{\alpha(t)}{1,\kappa,2}{\bz} \leq M.
\end{equation*}
In conclusion, one can pick some $t_0\in[T-\frac{T(M)}{2},T)$, where $T(M)$ is the lower bound for the existence interval of the solution, and continue the solution to at least $T+\frac{T(M)}{2}$, contradicting maximality. \qedhere
\end{proof}

\begin{defn}[Velocity moments]
Let $k\geq0$ be arbitrary. For $\alpha:\RZ\to\CN$ measurable we define the \textbf{$k$-th local velocity moment}
\begin{equation}
\moment{\alpha}{k}(\bx) \equiv \int_{\RV} \abs{\bv}^k \abs{\alpha(\bx,\bv)}^2 \intd\bv
\end{equation}
and the \textbf{$k$-th velocity moment}
\begin{equation}
\Moment{\alpha}{k} \equiv \int_{\RX} \moment{\alpha}{k}(\bx)~ \intd\bx = \int_{\RX} \int_{\RV} \abs{\bv}^k \abs{\alpha(\bx,\bv)}^2 \intd\bv~ \intd\bx.
\end{equation}
\end{defn}

\begin{lem}
\label{lem:moments}
\cite[Lem.1.8]{rein}
Let $k\geq 0$ be arbitrary and $\alpha: \RZ\to\CN$ be measurable. Now let $p\in[1,\infty]$ be some exponent and $0\leq l\leq k<\infty$. Define the exponent
\begin{equation*}
r \equiv \frac {k+\ddim\frac{p-1}{p}} {l+\frac{k-l}{p}+\ddim\frac{p-1}{p}}.
\end{equation*}
Then holds the estimate
\begin{equation}
\label{eqn:moment-estimate}
\Lpn{\moment{\alpha}{l}}{r}{\bx} \leq~ \cstMoment{k,l,p} \Lpn{\alpha}{2p}{\bz}^{\frac{2p(k-l)}{pk+(p-1)\ddim}} \Moment{\alpha}{k}^{\frac{lp+\ddim(p-1)}{pk+(p-1)\ddim}}.
\end{equation}
\end{lem}
\begin{proof}
Let $\bx$ be arbitrary and call $q\equiv \frac{p}{p-1}$ the conjugated exponent. Then we find for any $R>0$
\begin{align*}
\moment{\alpha}{l}(\bx) \leq& \int_{\ball{\bn}{R}} \abs{\bv}^l \abs{\alpha(\bx,\bv)}^2 \intd\bv + \int_{\RV-\ball{\bn}{R}} \abs{\bv}^l \abs{\alpha(\bx,\bv)}^2 \intd\bv \\
\leq& \Lpn{\alpha(\bx,\cdot)}{2p}{\bv}^2 \left(\int_{\ball{\bn}{R}} \abs{\bv}^{lq}\intd\bv\right)^{\frac 1q} + R^{l-k} \int_{\RV} \abs{\bv}^k \abs{\alpha(\bx,\bv)}^2 \intd\bv \\
\leq& \Lpn{\alpha(\bx,\cdot)}{2p}{\bv}^2 \left(\frac{\surS{\ddim}}{lq+\ddim}\right)^{\frac 1q} R^{l+\frac{\ddim}{q}} + \moment{\alpha}{k}(\bx)~ R^{l-k} \\
\stackrel{R\text{ opt.}}{=}& \underbrace{\left(\left(\frac{k-l}{l+\frac{\ddim}{q}}\right)^{\frac{l+\frac{\ddim}{q}}{k+\frac{\ddim}{q}}} + \left(\frac{l+\frac{\ddim}{q}}{k-l}\right)^{\frac{k-l}{k+\frac{\ddim}{q}}}\right) \left(\frac{\surS{\ddim}}{lq+\ddim}\right)^{\frac{k-l}{kq+\ddim}}}_{\equiv \cstMoment{k,l,p}} \Lpn{\alpha(\bx,\cdot)}{2p}{\bv}^{2\frac{k-l}{k+\frac{\ddim}{q}}} \moment{\alpha}{k}(\bx)^{\frac{l+\frac{\ddim}{q}}{k+\frac{\ddim}{q}}}.
\end{align*}
Now take the $r$-th power and integrate this expression, where the right-hand side is estimated by Hölder inequality once more.
\begin{align*}
\Lpn{\moment{\alpha}{l}^r}{1}{\bx} 
\leq&~ \cstMoment{k,l,p}^r \Lpn{\left(\Lpn{\alpha}{2p}{\bv}^{2p}\right)^{\frac{r}{p} \frac{k-l}{k+\frac{\ddim}{q}}} \moment{\alpha}{k}^{r\frac{l+\frac{\ddim}{q}}{k+\frac{\ddim}{q}}}}{1}{\bx} \\
=&~ \cstMoment{k,l,p}^r \Lpn{\left(\Lpn{\alpha}{2p}{\bv}^{2p}\right)^{\frac{k-l}{k-l+lp+(p-1)\ddim}} \moment{\alpha}{k}^{\frac{lp+\ddim(p-1)}{lp+(p-1)\ddim+k-l}}}{1}{\bx} \\
\leq&~ \cstMoment{k,l,p}^r \Lpn{\alpha}{2p}{\bz}^{\frac{2p(k-l)}{k-l+lp+(p-1)\ddim}} \Moment{\alpha}{k}^{\frac{lp+\ddim(p-1)}{lp+(p-1)\ddim+k-l}}.
\end{align*}
Upon taking the $r$-th root of this expression, we find the asserted inequality
\begin{equation*}
\Lpn{\moment{\alpha}{l}}{r}{\bx} \leq~ \cstMoment{k,l,p} \Lpn{\alpha}{2p}{\bz}^{\frac{2p(k-l)}{pk+(p-1)\ddim}} \Moment{\alpha}{k}^{\frac{lp+\ddim(p-1)}{pk+(p-1)\ddim}}. \quad \qedhere
\end{equation*}
\end{proof}

\begin{prop}[Globality criteria]
\label{prop:globality-criterion}
Let $\alpha\in\Cm{1}{[0,T)\times\RZ;\CN}$ be a solution of the Hamiltonian Vlasov-Poisson equation on a maximal interval of existence with characteristic tuple $(\rho,F,\varphi,K)$ and initial datum $\Bpn{\mathring\alpha}{1,\kappa,2}{\bz}<\infty$. If there is a continuous $h: \RNp \to \RNp$, s.t. either
\begin{enumerate}[label=(\roman*)]
\item $\Lpn{F(t)}{\infty}{\bx} \leq h(t)$ on $t\in[0,T)$,
\item $\Lpn{\rho(t)}{p}{\bx}\leq h(t)$ on $t\in[0,T)$ for some $p\in d\left[1-\frac{1}{\kappa},1\right)$, or
\item $\Moment{\alpha(t)}{k}\leq h(t)$ on $t\in[0,T)$ for some $k\geq \left[\left(1-\frac{1}{\kappa}\right)\ddim-1\right]\ddim \geq \left(\ddim-\frac{3}{2}\right)\ddim$,
\end{enumerate}
then $T=\infty$ and the solution is global.
\end{prop}
\begin{proof}
The idea is to replace the key bound $\globalBound{\alpha}{\kappa,2}$ in Lemma \ref{lem:local-solution} by a continuous function that remains finite for finite times. Therefore, we need to manipulate the estimate \eqref{eqn:key-bound}:
\begin{align*}
\Apn{\alpha(t)}{\kappa,2}{\bz} \stackrel{\eqref{eqn:key-bound}}{\leq}& \left(1+t+\int_{0}^{t} \Lpn{F(\tau)}{\infty}{\bx} e^\tau\intd\tau\right)^{\frac{\kappa}{2}} \Apn{\mathring\alpha}{\kappa,2}{\bz} \\
\leq& \left(1+t+\int_{0}^{t} \cstNewtonLp{p} \Lpn{\rho(\tau)}{p}{\bx}^{\frac{p}{\ddim}} \Lpn{\rho(\tau)}{\infty}{\bx}^{1-\frac{p}{\ddim}} e^\tau\intd\tau\right)^{\frac{\kappa}{2}} \Apn{\mathring\alpha}{\kappa,2}{\bz} \\
\leq& \left(1+t+\int_{0}^{t} \cstNewtonLp{p} \minR(\kappa,\ddim) \Lpn{\rho(\tau)}{p}{\bx}^{\frac{p}{\ddim}} e^\tau \Apn{\alpha(\tau)}{\kappa,2}{\bx}^{2\left(1-\frac{p}{\ddim}\right)}\intd\tau\right)^{\frac{\kappa}{2}} \Apn{\mathring\alpha}{\kappa,2}{\bz}.
\end{align*}
From the first inequality, (i) is obvious, from the last one, we see indeed, whenever $p\in\ddim\left[1-\frac{1}{\kappa},1\right)$ and (ii) holds, a Gronwall argument proves the existence of $g\in\Cm{0}{\RNp;\RNp}$, s.t., $\Apn{\alpha(t)}{\kappa,2}{\bz} \leq g(t)$. By replacing the general bound $\globalBound{\alpha}{\kappa,2}$ with the specific upper bound $g$ in the proof of Lemma \ref{lem:local-solution}, one infers that even $\Apn{\naZ\alpha(t)}{\kappa,2}{\bz}$ and therefore $\Bpn{\alpha(t)}{1,\kappa,2}{\bz}$ remains bounded on any bounded interval. By the Maximality Criterion \ref{prop:globality-criterion}, the solution is global.

If we now assume (iii), we can reduce this to the condition (i) by Lemma \ref{lem:moments}, i.e.,
\begin{align*}
\Lpn{\rho(t)}{p}{\bx} &= \Lpn{\moment{\alpha(t)}{0}}{p}{\bx} \stackrel{\text{Lem.\ref{lem:moments}}}{\leq} \cstMoment{k,l,q} \Lpn{\alpha(t)}{2q}{\bz}^{\frac{2qk}{qk+(q-1)\ddim}} \Moment{\alpha(t)}{k}^{\frac{(q-1)\ddim}{qk+(q-1)\ddim}} \\
&\leq \cstMoment{k,l,q} \Lpn{\mathring{\alpha}}{2q}{\bz}^{\frac{2qk}{qk+(q-1)\ddim}} \Moment{\alpha(t)}{k}^{\frac{(q-1)\ddim}{qk+(q-1)\ddim}},
\end{align*}
given that
\begin{align*}
\left(1-\frac{1}{\kappa}\right)\ddim \leq p = \frac{qk + (q-1)\ddim}{k+(q-1)\ddim} < \ddim.
\end{align*}
By adjusting $1\leq q\leq\infty$, we find that if $\Moment{\alpha(t)}{k}$ is bounded for some fixed $k\geq0$, then so is $\Lpn{\rho(t)}{p}{\bx}$ for any $1\leq p\leq 1+\frac{k}{\ddim}$. In order to satisfy the condition (ii), we find the natural restriction for $k$
\begin{equation*}
\left(1-\frac{1}{\kappa}\right)\ddim \leq 1+\frac{k}{\ddim}. \quad\qedhere
\end{equation*}
\end{proof}

\begin{thm}[Pfaffelmoser-Schaeffer adaptation]
\label{thm:pfaffelmoser-schaeffer-adaption}
Let $\ddim=3$, $\kappa\geq 2\ddim=6$ be fixed and $\mathring\alpha\in\Bp{1,\kappa,2}{\bz}$ be an initial datum, that in addition satisfies
\begin{enumerate}[label=(\roman*)]
\item $\Apn{\Lpn{\mathring{\alpha}}{\infty}{\bx}}{\ddim,2}{\bv}<\infty$,
\item $\exists\lambda\geq\ddim: \Apn{\Lpn{\naZ\mathring{\alpha}}{\infty}{\bx}}{\lambda,2}{\bv}<\infty$, and
\item $\Moment{\mathring{\alpha}}{2}<\infty$.
\end{enumerate}
Then the maximal solution for this initial datum is global.
\end{thm}
\begin{proof}
Let $\alpha$ be the maximal solution for the given initial datum and $f\equiv\abs{\alpha}^2$ its corresponding solution of the classical Vlasov-Poisson equation. It is easily seen that $\mathring{f}\equiv\abs{\mathring{\alpha}}^2$ then satisfies the conditions given in \cite[Gen.Ass.1]{pfaffelmoser}. \cite[Thm.19]{pfaffelmoser} in combination with \cite[Lem.10]{pfaffelmoser} proves that there is $h\in\Cm{0}{\RNp;\RNp}$, s.t., $\Lpn{F(t)}{\infty}{\bx}\leq h(t)$, yielding globality of $\alpha$ by Proposition \ref{prop:globality-criterion}-(i). \qedhere
\end{proof}

\begin{thm}[Lions-Perthame adaptation]
Let $\ddim=3$ be fixed, $\kappa\geq 2\ddim=6$, and $\mathring{\alpha}\in\Bp{1,\kappa,2}{\bz}$ be an initial datum, s.t. in addition $\Moment{\mathring{\alpha}}{k}<\infty$ for some $k>9\left(1-\frac{1}{\kappa}\right)-3\geq\frac{9}{2}$. Then the maximal solution for this initial datum is global.
\end{thm}
\begin{proof}
Let $\alpha$ be the maximal solution for the given initial datum and $f\equiv\abs{\alpha}^2$ be the corresponding classical solution of the Vlasov-Poisson system. \cite[Thm.1]{lionsperthame} proves that there is $h\in\Cm{0}{\RNp;\RNp}$, s.t. $\Moment{\alpha(t)}{k}\leq h(t)$, yielding globality for $\alpha$ by Proposition \ref{prop:globality-criterion}-(iii). \qedhere
\end{proof}

\begin{rmk}
\textbf{(i) Pfaffelmoser-Schaeffer.} While the relation between local-$\sup$ integrability conditions given in \cite[Gen.Ass.1]{pfaffelmoser} and the conditions (i), (ii) in Theorem \ref{thm:pfaffelmoser-schaeffer-adaption} are very close, it is worth noting the subtle difference. While $\alpha\in\Bp{1,\kappa,2}{\bz}$ only requires the integrability of the local supremum on balls in $\RZ$, conditions in \cite[Gen.Ass.1]{pfaffelmoser} require integrability of the local supremum on balls in $\RV$ and the full space $\RX$. Hence (i), (ii) are true restrictions, that are not genuinely fulfilled for $\alpha\in\Bp{1,\kappa,2}{\bz}$. \\

\noindent\textbf{(ii) Lions-Perthame.} In the Vlasov picture $\mathring{f}=\abs{\mathring{\alpha}}^2$, \cite[Thm.1]{lionsperthame} provides the global extension of solutions under the sole restriction of $\Moment{\mathring{\alpha}}{k}<\infty$ for some $k>3$. On the other hand, uniqueness of $f$ is only established for the case of $k>6$, but on a much larger class of solutions, essentially satisfying $\rho\in\Lp{\infty}{t,\text{loc}}\Lp{\infty}{\bx}$. We should remark, that any $\mathring{f}$ derived from $\mathring{\alpha}\in\Bp{1,\kappa,2}{\bz}$ is in $\Bp{1,\kappa,1}{\bz}$ and indeed satisfies all the regularity and integrability conditions of \cite[Thm.6]{lionsperthame}. Therefore on the Vlasov level, the solution $f$ is unique in that much larger class, as long as $k>6$. Nevertheless, as the map $\mathring{\alpha}\mapsto\abs{\mathring\alpha}^2$ is not one to one, the results are not instantly transferable to the Hamiltonian Vlasov setup.
\end{rmk}

\section{Acknowledgements}

I like to thank Markus Kunze for his support and many useful discussions and hints as well as Antti Knowles for a short discussion during his visit to our department. Also I like to thank the referee for his detailed and insightful comments on the paper.

\appendix

\section{Functions with integrable local supremum}

\label{chap:integrable-local-supremum}

This section provides some useful properties of the spaces $\Bp{k,\kappa,p}{\bz}$ as given in Definition \ref{defn:loc-sup-Lp}. As it will be useful for many estimates throughout this chapter, we define for any $a\geq b$ the constant
\begin{equation}
\minR(a,b) \equiv \inf_{R>0} \frac {(1+R)^a} {\volS{b} R^b} = \frac{1}{\volS{b}}\frac{a^a}{b^b} (a-b)^{b-a}.
\end{equation}

\begin{lem}
\label{lem:Ap}
For $d\in\NN$, $p\geq1$, and $\kappa\geq d$, we have:
\begin{enumerate}[label=(\roman*)]
\item $\Cp{0}{c} \subseteq \Ap{\kappa,p}{} \subseteq \Lp{p}{}\cap\Lp{\infty}{}$.
\item If $f\in\Lp{p}{}$ and $\nabla f\in\Ap{\kappa,p}{}$, then $f\in\Ap{\kappa+p,p}{}$.
\item For $\kappa\leq\lambda$, $\Apn{f}{\lambda,p}{}\leq \Apn{f}{\kappa,p}{}$. Hence, $\Ap{\kappa,p}{}\subseteq\Ap{\lambda,p}{}$.
\item For $\frac{1}{p}+\frac{1}{q}=1$, $\Apn{f\cdot g}{\frac{\kappa}{p}+\frac{\lambda}{q},1}{}\leq \Apn{f}{\kappa,p}{}\Apn{g}{\lambda,q}{}$, i.e., a generalized Hölder inequality holds.
\end{enumerate}
\end{lem}
\begin{proof}
\textbf{(i).} Let $f\in\Cp{0}{c}$ have compact support. Then there is $R_0>0$, s.t. $\supp f\subseteq\ball{\bn}{R_0}$. Thus,
\begin{align*}
\Apn{f}{\kappa,p}{} =& \sup_{R\geq 0} (1+R)^{-\frac{\kappa}{p}} \left(\int_{\RN^d} \sup_{\abs{\bar{\bz}}\leq R} \abs{f(\bz+\bar{\bz})}^p~\intd\bz\right)^{\frac 1p} \\
\leq& \Lpn{f}{\infty}{} \sup_{R\geq 0} (1+R)^{-\frac{\kappa}{p}} \left(\int_{\RN^d} \sup_{\abs{\bar{\bz}}\leq R} \unity_{\ball{\bn}{R_0}}(\bz+\bar{\bz})\intd\bz \right)^{\frac{1}{p}} \\
=& \Lpn{f}{\infty}{} \sup_{R\geq 0} (1+R)^{-\frac{\kappa}{p}} (R_0+R)^{\frac{d}{p}} \volS{d}^{\frac{1}{p}} < \infty.
\end{align*}
$\Lpn{f}{p}{} \leq \Apn{f}{\kappa,p}{}$ is obvious by taking $R=0$. In addition we find for any $R>0$,
\begin{align*}
\Lpn{f}{\infty}{} \leq& \left(\frac{1}{\volS{d}R^d} \int_{\RN^d} \sup_{\abs{\bar{\bz}}\leq R} \abs{f(\bz+\bar{\bz})}^p~\intd\bz\right)^{\frac 1p} \leq \left(\frac{\left(1+R\right)^\kappa}{\volS{d}R^d}\right)^{\frac 1p} \Apn{f}{\kappa,p}{} \\
\stackrel{R\text{ opt.}}{=}& \minR(\kappa,d)^{\frac 1p} \Apn{f}{\kappa,p}{}.
\end{align*}
\\

\noindent\textbf{(ii).} Let $f\in\Lp{p}{}$ and $\nabla f\in\Ap{\kappa,p}{}$, then for any $R>0$
\begin{align*}
(1+R)^{-\frac{\kappa+p}{p}} \left(\int_{\RN^d} \sup_{\abs{\bar{\bz}}\leq R} \abs{f(\bz+\bar{\bz})}^p \intd\bz\right)^{\frac 1p} \leq& (1+R)^{-\frac{\kappa+p}{p}} \left(\int_{\RN^d} \left(\abs{f(\bz)} + R \sup_{\abs{\bar{\bz}}\leq R} \abs{\nabla f(\bz+\bar{\bz})}\right)^p\intd\bz\right)^{\frac 1p} \\
\leq& \Lpn{f}{p}{} + \Apn{\nabla f}{\kappa,p}{}.
\end{align*}
\\

\noindent\textbf{(iii).} Obvious. \\

\noindent\textbf{(iv).} For $R>0$, we have
\begin{align*}
&(1+R)^{-\frac{\kappa}{p}-\frac{\lambda}{q}} \int_{\RN^d} \sup_{\abs{\bar{\bz}}\leq R} \abs{f(\bz+\bar{\bz})~g(\bz+\bar{\bz})} \intd\bz \\
\leq& (1+R)^{-\frac \kappa p - \frac \lambda q} \int_{\RN^d} \left(\sup_{\abs{\bar{\bz}}\leq R} \abs{f(\bz+\bar{\bz})}\right) \left(\sup_{\abs{\bar{\bz}}\leq R} \abs{g(\bz+\bar{\bz})}\right)\intd\bz 
\stackrel{\text{Hölder}}{\leq} \Apn{f}{\kappa,p}{} \Apn{g}{\lambda,q}{}. \quad \qedhere
\end{align*}
\end{proof}

\begin{lem}
\label{lem:compact-dense}
Let $k\in\NN$, $\kappa\geq d$, and $p\geq 1$ be arbitrary. Then the functions $\Cp{k}{c}$ of compact support are dense in $\Bp{k,\kappa,d}{}$.
\end{lem}
\begin{proof}
\textbf{(i) $\Cp{k}{c}\subseteq\Bp{k,\kappa,p}{}$.} As all derivatives have compact support, this is implied by Lemma \ref{lem:Ap}-(i). \\

\noindent\textbf{(ii) $\Bp{k,\kappa,p}{}\subseteq\overline{\Cp{k}{c}}$.} Let $f\in\Bp{k,\kappa,p}{}$ be some function. Let $\eta:\RN^d \to [0,1]$ be some $\Cp{k}{c}$ function, such that $\eta \equiv 1$ on $\ball{\bn}{1}$ and $\eta\equiv 0$ on $\ball{\bn}{2}^c$. For any $\epsilon>0$ we then define $\eta_\epsilon(\bx)\equiv \eta(\epsilon\bx)$. Obviously, $f\eta_\epsilon\in\Cp{k}{c}$. We want to show that for any $\alpha\in\NN_0^d, \abs{\alpha}\leq k$, $\Apn{\Del{\alpha}(f\eta_\epsilon)-\Del{\alpha}f}{\kappa,p}{} \to 0$ for $\epsilon\downarrow 0$. In fact, we have by the Leibniz rule,
\begin{align*}
\Del{\alpha}(f\eta_\epsilon) =& \sum_{\beta\in\NN_0^d, \beta\leq\alpha} \binom{\alpha}{\beta} \left(\Del{\alpha-\beta}f\right) \left(\Del{\beta}\eta_\epsilon\right) = \left(\Del{\alpha}f\right) \eta_\epsilon + \sum_{\beta\in\NN_0^d, \beta\leq\alpha, \beta\neq\bn} \binom{\alpha}{\beta} \left(\Del{\alpha-\beta}f\right) \left(\Del{\beta}\eta_\epsilon\right)
\end{align*}
with the usual notation $\binom{\alpha}{\beta}\equiv\prod_i \binom{\alpha_i}{\beta_i}$. The first term converges to $\Del{\alpha}f$ and as $\left(\Del{\beta}\eta_\epsilon\right)(\bx) = \epsilon^{\abs{\beta}} \left(\Del{\beta}\eta\right)(\epsilon\bx)$, all the other terms are integrably dominated for bounded $\epsilon$ and their $\Ap{\kappa,p}{}$ norms vanish in the limit $\epsilon\downarrow 0$. This proves $\lim_{\epsilon\downarrow 0} \Apn{\Del{\alpha}(f\eta_\epsilon)-\Del{\alpha}f}{\kappa,p}{}=0$. \qedhere
\end{proof}

\begin{thm}
\label{thm:ap-banach}
For any $k\in\NN_0$, $\kappa\geq d$, and $p\geq1$ $\Bp{k,\kappa,p}{}$ is a Banach space.
\end{thm}
\begin{proof}
It remains to prove the completeness. Let $\{f_n\}$ be a Cauchy sequence in $\Bp{k,\kappa,p}{}$. Because $\Ap{\kappa,p}{}$ embeds continuously into $\Cp{0}{}$, $\Bp{k,\kappa,p}{}$ embeds into $\Cp{k}{}$. Therefore, $\{f_n\}$ is Cauchy in $\Cp{k}{}$ and possesses a uniform limit in this Banach space, say $f$.

We now want to show that $\Bpn{f_n-f}{k,\kappa,p}{}\to 0$ and therefore $f\in\Bp{k,\kappa,p}{}$ is the right limit. It suffices to check the case $k=0$, as all the arguments can be repeated for the derivatives.

Let $\epsilon>0$ be arbitrary and choose $N\in\NN$, s.t. by the Cauchy property for $n,m\geq N$ $\Apn{f_n-f_m}{\kappa,p}{}\leq \epsilon$. Now let $\bx\in\RN^d$ and we find by the reverse triangle inequality for the supremum norm on $\ball{\bx}{R}\subseteq\RN^d$
\begin{align*}
\abs{\sup_{\abs{\bar{\bx}}\leq R} \abs{f_n(\bx+\bar{\bx})-f_m(\bx+\bar{\bx})} - \sup_{\abs{\bar{\bx}}\leq R} \abs{f(\bx+\bar{\bx})-f_m(\bx+\bar{\bx})}} \leq& \sup_{\abs{\bar{\bx}}\leq R} \abs{f_n(\bx+\bar{\bx})-f(\bx+\bar{\bx})} \\
\leq& \Lpn{f_n-f}{\infty}{\bx} \stackrel{n\to\infty}{\to} 0,
\end{align*}
in other words,
\begin{align*}
\lim_{n\to\infty} \sup_{\abs{\bar{\bx}}\leq R} \abs{f_n(\bx+\bar{\bx})-f_m(\bx+\bar{\bx})} = \sup_{\abs{\bar{\bx}}\leq R} \abs{f(\bx+\bar{\bx})-f_m(\bx+\bar{\bx})}.
\end{align*}
Now, for any $R>0$, using this equation at every $\bx$ and in combination with Fatou's Lemma, we derive for any $m\geq N$
\begin{align*}
&\frac{1}{(1+R)^{\frac{\kappa}{p}}} \left(\int_{\RN^d} \sup_{\abs{\bar{\bx}}\leq R} \abs{f(\bx+\bar{\bx})-f_m(\bx+\bar{\bx})}^p \intd\bx\right)^{\frac 1p} \\
=& \frac{1}{(1+R)^{\frac{\kappa}{p}}} \left(\int_{\RN^d} \liminf_{n\to\infty} \sup_{\abs{\bar{\bx}}\leq R} \abs{f_n(\bx+\bar{\bx})-f_m(\bx+\bar{\bx})}^p \intd\bx\right)^{\frac 1p} \\
\leq& \liminf_{n\to\infty} \frac{1}{(1+R)^{\frac{\kappa}{p}}} \left(\int_{\RN^d} \sup_{\abs{\bar{\bx}}\leq R} \abs{f_n(\bx+\bar{\bx})-f_m(\bx+\bar{\bx})}^p \intd\bx\right)^{\frac 1p} \\
\leq& \liminf_{n\to\infty} \Apn{f_n-f_m}{\kappa,p}{} \leq \epsilon.
\end{align*}
Therefore, $\Apn{f-f_m}{\kappa,p}{} \leq\epsilon$ and $f\in\Ap{\kappa,p}{}$ is indeed the limit. \qedhere
\end{proof}

\section{Inequalities}

\begin{lem}[Newtonian potential]
\label{lem:newtonian-ineq}
Define the Newtonian interaction potential in space dimension $\ddim\geq3$ by
\begin{equation}
\label{eqn:newtonian-interaction-potential}
\newtonian{\bx} \equiv \frac{\abs{\bx}^{2-\ddim}}{\surS{\ddim}\ddim\left(2-\ddim\right)}.
\end{equation}
Let $\rho\in\Cp{0,\alpha}{\bx}\cap\Lp{1}{\bx}$ for some $\alpha\in(0,1]$ be Hölder continuous and $U_\rho:\RX\to\RN$ be given by
\begin{align*}
U_\rho(\bx) = \convol{\newtonian{\cdot}}{\rho}(\bx) = \int_{\RX} \newtonian{\bx-\by}~\rho(\by)~\intd\by.
\end{align*}
Then $U_\rho\in\Cp{2,\alpha}{\bx}$ and the following representation formulae hold for any $\bx_0\in\RX$ and $\abs{\bx-\bx_0} < \frac R2$:
\begin{align}
\label{eqn:DU-formula}
\partial_j U_\rho(\bx) =& \int_{\RX} \partial_j\newtonian{\bx-\by}~\rho(\by)~\intd\by, \\
\nonumber
\partial_i\partial_j U_\rho(\bx) =& \int_{\RX-\ball{\bx_0}{R}} \partial_i\partial_j \newtonian{\bx-\by}~\rho(\by)~\intd\by + \int_{\ball{\bx_0}{R}} \partial_i\partial_j \newtonian{\bx-\by}\left(\rho(\by)-\rho(\bx)\right)~\intd\by \\
\label{eqn:D2U-formula}
&- \rho(\bx)\int_{\partial\ball{\bx_0}{R}} \partial_j\newtonian{\bx-\by} ~\nu_i(\by)~\intd s(\by).
\end{align}
In addition, for any $\bx_1,\bx_2\in\RX$ we find the estimate
\begin{equation}
\label{eqn:D2U-Hoelder-estimate}
\abs{\naX^2 U_\rho(\bx_1)-\naX^2 U_\rho(\bx_2)} \leq L(\ddim,\alpha) \sup_{\xi_1\neq\xi_2} \frac{\abs{\rho(\xi_1)-\rho(\xi_2)}}{\abs{\xi_1-\xi_2}^\alpha} \abs{\bx_1-\bx_2}^\alpha
\end{equation}
In particular, $U_\rho\in\Cp{2,\alpha}{\bx}$ and we find the estimates:
\begin{equation}
\label{eqn:DU-estimate}
\forall p\in[1,\ddim)~\exists \cstNewtonLp{p}>0: \Lpn{\naX U_\rho}{\infty}{\bx} \leq \cstNewtonLp{p} \Lpn{\rho}{\infty}{\bx}^{1-\frac p\ddim}~\Lpn{\rho}{p}{\bx}^{\frac p\ddim}.
\end{equation}
Finally, in the case $\alpha=1$, there is $\cstNewtonLn>0$, s.t. for every $0<R\leq r$
\begin{equation}
\label{eqn:D2U-estimate}
\Lpn{\naX^2 U_\rho}{\infty}{\bx} \leq \cstNewtonLn\left[\left(1+\ln\frac rR\right)\Lpn{\rho}{\infty}{\bx}+\frac{1}{r^\ddim} \Lpn{\rho}{1}{\bx}+R\Lpn{\naX\rho}{\infty}{\bx}\right]
\end{equation}
\begin{equation}
\label{eqn:D2U-simp-estimate}
\Lpn{\naX^2 U_\rho}{\infty}{\bx} \leq \cstNewtonLn\left[\left(1+\Lpn{\rho}{\infty}{\bx}\right)\left(1+\ln_+\Lpn{\naX\rho}{\infty}{\bx}\right)+\Lpn{\rho}{1}{\bx}\right],
\end{equation}
where $\ln_+ x = \ln\max\{1,x\}$.
\end{lem}
\begin{proof}
The formulae are standard results from potential theory, see for example \cite[Chapter 4]{gilbargtrudinger}. \qedhere
\end{proof}

\section*{Conflict of interest statement}

I declare that there are no conflicts of interest, because this work has not been funded by third parties.

\bibliographystyle{abbrv}

\end{document}